\documentclass[12pt]{article}

\usepackage[T1]{fontenc}
\usepackage{lmodern}
\usepackage[utf8]{inputenc}
\usepackage[margin=1.1in]{geometry}
\usepackage{graphicx}
\graphicspath{{Figures/}}
\usepackage{amsmath}
\usepackage{amssymb}
\usepackage{amsthm}
\usepackage[authoryear,round]{natbib}
\usepackage{tabularx}
\usepackage{siunitx}
\usepackage{placeins}
\usepackage{booktabs}
\usepackage[colorlinks=true,linkcolor=blue,citecolor=blue,urlcolor=blue]{hyperref}
\hypersetup{
  pdftitle={Correlation Matrices in High Dimensions: Volume, Spectrum, and Extremes},
  pdfauthor={Peter Reinhard Hansen},
  pdfsubject={Correlation matrices, elliptope, random matrix theory},
  pdfkeywords={correlation matrices, correlation networks, elliptope, extreme-value theory, LKJ, Marchenko-Pastur, nearest correlation matrix, Poisson point process},
}
\newcolumntype{Y}{>{\centering\arraybackslash}X}
\theoremstyle{plain}
\newtheorem{theorem}{Theorem}
\newtheorem{lemma}{Lemma}
\newtheorem{corollary}{Corollary}
\newtheorem{proposition}{Proposition}

\theoremstyle{definition}
\newtheorem{remark}{Remark}

\title{Correlation Matrices in High Dimensions:\\ Volume, Spectrum, and Extremes}
\author{Peter Reinhard Hansen\thanks{Department of Economics, University of North Carolina at Chapel Hill, 107 Gardner Hall, Chapel Hill, NC 27599-3305. E-mail: \texttt{hansen@unc.edu}.}}
\date{October 5, 2026}

\begin{document}
\maketitle

\begin{abstract}
\noindent The set of $n\times n$ correlation matrices, the elliptope $\mathcal{E}_n$, has volume decaying at the super-exponential rate $\exp\{-\tfrac14 n^2\log n\}$. We characterize where this vanishing volume concentrates. 
A uniform draw is entrywise close to the identity yet globally far from it and nearly singular. Its maximum absolute correlation is of order $\sqrt{\log n/n}$, its Frobenius distance is asymptotic to $\sqrt n$, its empirical spectral distribution converges to the Marchenko--Pastur law with ratio one, and its smallest eigenvalue has the exact $\operatorname{Beta}(1,d)$ distribution, where $d=n(n-1)/2$, and is therefore of order $n^{-2}$.
Distinct off-diagonal entries are exactly pairwise independent under every $\operatorname{LKJ}(\eta)$ law, which yields a Chen--Stein proof of the extreme-correlation point-process limit for the whole LKJ family and an $O(n^{-1})$ total-variation bound for finite-dimensional exceedance counts relative to Poisson laws with their exact finite-$n$ means. Finally, for a bounded, centered i.i.d.\ off-diagonal perturbation of any correlation matrix, the nearest-correlation projection removes a fraction of the squared perturbation that tends to one and recovers the original matrix in average squared error per entry.

\textbf{MSC2020 subject classifications:} Primary 60B20, 62H10; secondary 15B48, 52A20, 60G70.
\end{abstract}

\section{Introduction}

Correlation matrices play a central role in multivariate statistics, econometrics, and applied probability.
Despite their importance, the geometry of the set of all $n\times n$ correlation matrices,
$$
\mathcal{E}_{n}=\{ C\in\mathbb{R}^{n\times n} : C=C^{\prime}, C\succeq 0, C_{ii}=1 \},
$$
is relatively unexplored in high dimensions. The positive semidefinite constraint is highly nonlinear. It couples all off-diagonal elements, while the unit diagonal fixes an affine subspace of dimension $d=n(n-1)/2$. An intersection of the positive semidefinite (PSD) cone $\mathbb{S}_{+}^{n}$ with an affine subspace is a spectrahedron, and the particular spectrahedron that arises from unit diagonal elements is called the elliptope.
The exact formula for $\operatorname{Vol}(\mathcal{E}_n)$ is classical and has been used mainly to bound the failure rates of rejection-sampling algorithms. Proposition~\ref{prop:VolumeCn} restates it together with a refined asymptotic expansion, equivalent to that of \citet{BohmHornik:2014}, in the natural dimension $d=n(n-1)/2$, and Section~\ref{sec:volume} reviews its history.

The volume and structure of $\mathcal{E}_{n}$ matter for several reasons.
First, they help explain why empirical correlation matrices are noisy, ill-conditioned, and strongly shaped by random-matrix effects whenever $n$ is large relative to the sample size $T$. When $T\asymp n$, the accumulated entrywise sampling noise produces a Frobenius displacement from $I_n$ of order $\sqrt{n}$, which is the global scale of the elliptope (Proposition~\ref{prop:C-I}). As $n/T\rightarrow1$ the average precision also diverges (Proposition~\ref{prop:precision}), which is why minimum-variance portfolio weights become unstable as the sample size approaches the dimension.
Second, the volume of $\mathcal{E}_{n}$ quantifies how rapidly semidefinite feasible sets shrink in high dimensions. The elliptope has only $n$ linear constraints beyond positive semidefiniteness, yet it occupies a super-exponentially small fraction of the cube $[-1,1]^d$. Related geometric constraints appear in machine learning, including PSD-constrained network layers \citep{HuangVanGool:2017} and correlation-parameterized kernels for categorical inputs in Gaussian process models \citep{RoustantEtAl2020}.
Third, because of this volume collapse, naive parameterizations of correlation matrices, such as choosing the off-diagonal entries independently in $[-1,1]$, almost never yield a positive semidefinite matrix when $n$ is large. In practice, one therefore relies on parameterizations that enforce positive semidefiniteness by construction, such as the Generalized Fisher Transformation (GFT) by \citet{ArchakovHansen:Correlation}, partial correlations \citep{Joe:2006}, vines \citep{LewandowskiKurowickaJoe:2009}, or hyperspherical angles \citep{PourahmadiWang:2015}; see also \citet{ArchakovHansenLuo-RandomCorr:2024}.
Fourth, in Bayesian modeling, prior distributions over correlation matrices have the elliptope as their support. The separation strategy of \citet{BarnardMcCullochMeng:2000} models a covariance matrix as $\Sigma=DCD$, with independent priors on the standard deviations $D$ and the correlation matrix $C$, so the prior on $C$ is a distribution on $\mathcal{E}_n$. Its geometry then matters for prior elicitation, because volume-based priors, including fixed-$\eta$ LKJ laws, shrink individual correlations toward zero and place most of their mass near the singular boundary (Theorems~\ref{thm:specfloor} and~\ref{thm:LKJscaling}).

This paper characterizes where the vanishing volume of $\mathcal{E}_n$ lies and what this implies for statistical practice. Two elementary facts drive the analysis. The first is the homothety of Theorem~\ref{thm:specfloor}, by which the correlation matrices whose smallest eigenvalue is at least $x$ form the scaled copy $xI_n+(1-x)\mathcal{E}_n$ of the elliptope. It gives the exact law $\lambda_{\min}(C)\sim\operatorname{Beta}(1,d)$ under the uniform distribution, which to our knowledge is new, and it reduces Frobenius projection onto correlation matrices with a prescribed spectral floor, by affine rescaling, to the ordinary nearest-correlation problem (Corollary~\ref{cor:floorprojection}). The second is the Wishart/LKJ identity (Proposition~\ref{prop:WishartGram}), which yields the known sample-correlation representation of the uniform distribution on $\mathcal{E}_n$, namely the law of the uncentered sample correlation matrix of $n$ independent Gaussian variables observed $T=n+1$ times (equivalently, the Pearson correlation matrix at $T=n+2$), which we exploit systematically. The resulting geometry has two scales. Almost all of $\mathcal{E}_n$ lies entrywise close to the identity matrix, with the largest correlation of order $\sqrt{\log n/n}$ (Proposition~\ref{prop:entrywise} and Corollary~\ref{cor:gumbel}), yet the Frobenius distance from $I_n$ grows like $\sqrt{n}$ (Proposition~\ref{prop:C-I}), the empirical eigenvalue distribution converges to the Marchenko--Pastur law with ratio one (Theorem~\ref{thm:MP}), and the smallest eigenvalue is of order $n^{-2}$.

There is an extensive literature on extreme values of sample correlations. For sample correlation matrices of independent Gaussian variables, the Gumbel limit of the largest off-diagonal entry was established by \citet{Jiang:2004maxentry}. Pairwise independence of distinct sample correlations under Gaussian and spherical null models is classical and follows from the representation of the sample correlations as inner products of independent vectors uniform on the sphere, which \citet{HeroRajaratnam:2011} use to derive the mean number and Poisson behavior of threshold exceedances in their correlation-screening framework. Poisson point-process convergence for the off-diagonal entries of sample covariance and correlation matrices is developed in generality by \citet[theorem~3.7]{HeinyMikoschYslas:2021}. Via the sample-correlation representation, these results apply at $T=n+2$ and therefore already cover the uniform law on $\mathcal{E}_n$. Our contribution on the extremes is threefold. First, pairwise independence is implicit for the entire LKJ family in its partial-correlation vine construction \citep{Joe:2006,LewandowskiKurowickaJoe:2009}, where the independence of the first-row correlations is stated explicitly \citep{JoeKurowicka:2026}. Lemma~\ref{lem:pairwise} is the explicit all-pairs formulation, with a direct Bartlett proof for every real degree of freedom, and the vine construction extends it to mutual independence along any forest. Second, exact pairwise independence gives a Chen--Stein proof of the Poisson point-process limit for every LKJ law (Theorem~\ref{thm:PPP}). The joint exceedance probabilities of edges sharing a vertex factor exactly, so the error terms in the Chen--Stein bound are of order $n^{-1}$. The same bound controls the finite-dimensional exceedance counts in total variation, at rate $O(n^{-1})$, relative to Poisson laws with their exact finite-$n$ means. Third, the uniform law serves as a null model for thresholded correlation networks in which every edge comes from one positive semidefinite matrix. At extreme thresholds it agrees with the Erd\H{o}s--R\'enyi benchmark, and at the moderate thresholds used in practice it gives the exact inflation of triangle counts relative to that benchmark (Section~\ref{sec:extremes}).

The remaining results concern LKJ laws and repair. Theorem~\ref{thm:LKJscaling} describes the LKJ$(\eta_n)$ law when $\eta_n$ grows with $n$. The limiting spectrum changes only when $\eta_n$ grows at least as fast as $n$, and the Frobenius distance from the identity stays bounded only when $\eta_n$ grows at least as fast as $n^2$. The second threshold follows from the Beta marginal of the LKJ law; the first is not visible from the density and comes from the sample-correlation representation. Theorem~\ref{thm:nearest} concerns repair. Add bounded, centered i.i.d.\ noise to the off-diagonal entries of any correlation matrix and project the result onto $\mathcal{E}_n$. The squared distance moved by the projection, relative to the squared Frobenius norm of the noise, converges to one, and the projected matrix is consistent for the original in average squared error per entry. The reason is a mismatch of norms: the nuclear norm of a correlation matrix is its trace, $n$, so no correlation matrix can align with a perturbation whose operator norm is of order $\sqrt{n}$ while its Frobenius norm is of order $n$. The same trace bound gives the rate $n^{-1/2}$ for a different estimator, universal singular value thresholding \citep{Chatterjee:2015}. Theorem~\ref{thm:nearest} establishes the same rate for the nearest-correlation projection itself, almost surely, together with the limit of the share of the perturbation that is removed. Proposition~\ref{prop:identitycenter} shows that the identity is the most robust target for sign-symmetric log-concave perturbations, and for entrywise cubes around $I_n$ three scales separate: half-width $1/(n-1)$ for the whole cube to consist of correlation matrices, order $n^{-1/2}$ for most of it to do so, and order $\sqrt{\log n/n}$ for the cube to contain most of the elliptope.

\section{Theoretical Results on Volume}\label{sec:volume}
The set of valid $n\times n$ correlation matrices $\mathcal{E}_n$ is a subset of the hypercube $[-1,1]^{d}$, 
where $d=n(n-1)/2$. 
We will use both the $n$- and $d$-parameterizations as convenient, where the inverse relation is $n=\frac{1+\sqrt{1+8d}}{2}$. All limits and asymptotic statements are as $n\rightarrow\infty$, equivalently $d\rightarrow\infty$, unless stated otherwise. For a matrix, $\|\cdot\|$ denotes the operator norm and $\|\cdot\|_F$ the Frobenius norm.

The exact formula (\ref{eq:VolCn}) is a classical result, first derived by \citet{JohnsonNaevdal:1998} via the Schur parametrization of positive semidefinite matrices; equivalent expressions were subsequently obtained by \citet{Joe:2006} and \citet{PourahmadiWang:2015}, and rediscovered via the hypersphere decomposition by \citet{EastmanHollisNumpacharoenSchlieper:2016}, extending \citet{RousseeuwMolenberghs:1994}; see \citet{ForresterZhang:2020} for a survey. The following proposition states this classical formula together with the asymptotic expansion (\ref{eq:VolCnAsym}). Let $\zeta(s)$ denote the Riemann zeta function; $\zeta(-1)=-1/12$ and $\zeta^\prime(-1)\approx -0.1654$.

\begin{proposition}\label{prop:VolumeCn}The Lebesgue volume of $\mathcal{E}_n$ satisfies the recursion
$$
\operatorname{Vol}(\mathcal{E}_{n+1})
=\operatorname{Vol}(\mathcal{E}_{n})\times
\left[B(\tfrac{n+1}{2},\tfrac{1}{2})\right]^{n},\qquad n\geq1,
$$
with $\operatorname{Vol}({\mathcal{E}_{1}})\equiv 1$; consequently,
\begin{equation}\label{eq:VolCn}
\operatorname{Vol}(\mathcal{E}_n)
=\pi^{\frac{n(n-1)}{4}}
\prod_{j=1}^{n}\frac{\Gamma\big(\tfrac{j+1}{2}\big)}{\Gamma\big(\tfrac{n+1}{2}\big)},
\qquad\text{for }n\geq 1.
\end{equation}
For large $n$, $\log\operatorname{Vol}(\mathcal{E}_n)=-\tfrac{n^2}{4}\log \frac{n}{2\pi\sqrt{e}}+O(n\log n)$ and in terms of $d=\tfrac{n(n-1)}{2}$ we have
\begin{equation}\label{eq:VolCnAsym}
\log\operatorname{Vol}(\mathcal{E}_n)
=-\tfrac{d}{4}\log(\tfrac{d}{2e\pi^2})-\tfrac{1}{2\sqrt{2}}\sqrt{d}-\tfrac{1}{48}\log d + \kappa
+O\big(\tfrac{1}{\sqrt{d}}\big), 
\end{equation}
where $\kappa=\tfrac{1}{4}\left[(1+\zeta(-1))\log 2-\zeta(-1)+2\zeta^\prime(-1)\right]\approx 0.097$. 
\end{proposition}
The expansion (\ref{eq:VolCnAsym}) is a reparameterization of that of \citet{BohmHornik:2014}, who obtained it in their analysis of rejection sampling and state it in the $n$-parameterization for the hypercube probability $\operatorname{Vol}(\mathcal{E}_n)/2^d$; their constant involves the Glaisher--Kinkelin constant $A$, which is equivalent to $\kappa$ via $\log A=\tfrac{1}{12}-\zeta^\prime(-1)$. The form (\ref{eq:VolCnAsym}) is compact in the intrinsic dimension and accurate even for moderate $n$: its absolute error is below $0.02$ already at $n=5$ and below $10^{-3}$ by $n=100$, in line with the $O(d^{-1/2})$ remainder, whereas the leading-order term alone has absolute error that diverges with $n$; see Figure~\ref{fig:VolCorr} in the Supplement. Exact closed-form expressions and numerical values for selected dimensions up to $n=30$ are given in Table~\ref{tab:CnVolume} in the Supplement. For perspective: at $n=10$ the feasible set occupies roughly $2\times10^{-14}$ of the hypercube $[-1,1]^{45}$; at $n=20$ this drops to roughly $4\times10^{-86}$; and by $n=30$ to below $10^{-233}$. The volume of $\mathcal{E}_n$ is that of an entrywise cube of half-width $\tfrac12\operatorname{Vol}(\mathcal{E}_n)^{1/d}=(e\pi^2/4)^{1/4}n^{-1/2}(1+o(1))\approx1.61/\sqrt{n}$, the scale of a typical entry, whose standard deviation is $1/\sqrt{n+1}$, rather than the scale $1/(n-1)$ of the largest cube around $I_n$ inside $\mathcal{E}_n$ (Section~\ref{sec:nearest}).

\section{Concentration, Spectral Structure, and Extreme Correlations}\label{sec:concentration}

While the volume of the elliptope shrinks rapidly, we now show where this vanishing volume concentrates. The answer depends on the norm, because almost all of $\mathcal{E}_n$ is entrywise close to $I_n$ while being nearly singular and far from $I_n$ in the Frobenius norm.

\subsection{Entrywise and Frobenius Concentration}\label{sec:twoscale}

For $r\in[0,1]$ we define
$$\mathcal{E}_n(r)=\{C\in \mathcal{E}_n:\max_{i<j}|C_{ij}|\leq r\},$$
the part of $\mathcal{E}_n$ within max-norm distance $r$ of $I_n$, so that $\mathcal{E}_n(0)=\{I_n\}$ and $\mathcal{E}_n(1)=\mathcal{E}_n$.

\begin{proposition} \label{prop:entrywise}
    For any fixed $r\in(0,1)$,
    $$\log\left(1-\frac{\operatorname{Vol}(\mathcal{E}_n(r))}{\operatorname{Vol}(\mathcal{E}_n)}\right)=-\frac{n}{2}\left|\log(1-r^2)\right|+O(\log n),\quad\text{as }n\rightarrow\infty.$$
\end{proposition}

So, the fraction of $\operatorname{Vol}(\mathcal{E}_n)$ that lies outside $\mathcal{E}_n(r)$ vanishes at the rate $\exp\{-\tfrac{n}{2}|\log(1-r^2)|\}$, up to a polynomial factor in $n$. This follows from the Beta tail of a single off-diagonal entry, $p_n(r)=\Pr(|C_{12}|>r)$, combined with Boole's inequality. The proposition shows that almost all high-dimensional correlation matrices are in the vicinity of the identity matrix. The convergence is slow, however, as the following typical-value calculation for the largest correlation coefficient shows.

Let $N_n(r)=\#\{i<j:|C_{ij}|>r\}$. By linearity of expectation and the symmetry of the uniform distribution, $\mathbb{E}N_n(r)=d\Pr(|C_{ij}|>r)$ exactly. Note that $\Pr(\max_{1\leq i<j\leq n}|C_{ij}|>r)=\Pr(N_n(r)\geq 1)$. We can select $r$ such that $\mathbb{E}N_n(r)\approx 1$, by solving $\Pr(|C_{ij}|>r)=1/d$. Under the uniform distribution on $\mathcal{E}_n$ we have $(C_{ij}+1)/2\sim \operatorname{Beta}(\tfrac{n}{2},\tfrac{n}{2})$, so $\operatorname{var}(C_{ij})=\tfrac{1}{n+1}$ and $\sqrt{n+1}C_{ij}\overset{d}{\rightarrow}N(0,1)$ as $n\rightarrow\infty$; we use this Gaussian approximation to evaluate $\Pr(|C_{ij}|>r)$. 

Next, we use the standard extreme-value expansion: the solution to $\Phi(-q)=\tfrac{1}{m}$ is
$$q=\sqrt{2\log m} -\frac{\log\log m +\log4\pi}{2\sqrt{2\log m}} + o\left(\tfrac{1}{\sqrt{\log m}}\right).$$
With $m=2d\approx n^2$ and $\sqrt{n+1}r =  q$
we arrive at 
\begin{equation}
M_n\equiv\max_{1\leq i<j\leq n}|C_{ij}|\approx\left(\frac{1}{n+1}\log\frac{n^4}{8\pi\log n}\right)^{1/2}.
\label{eq:maxrhoapprox}
\end{equation}

The approximation (\ref{eq:maxrhoapprox}) is shown in Figure~\ref{fig:CorrNearIdentity}.
\begin{figure}[!htb]
\centering{}\includegraphics[width=0.7\textwidth]{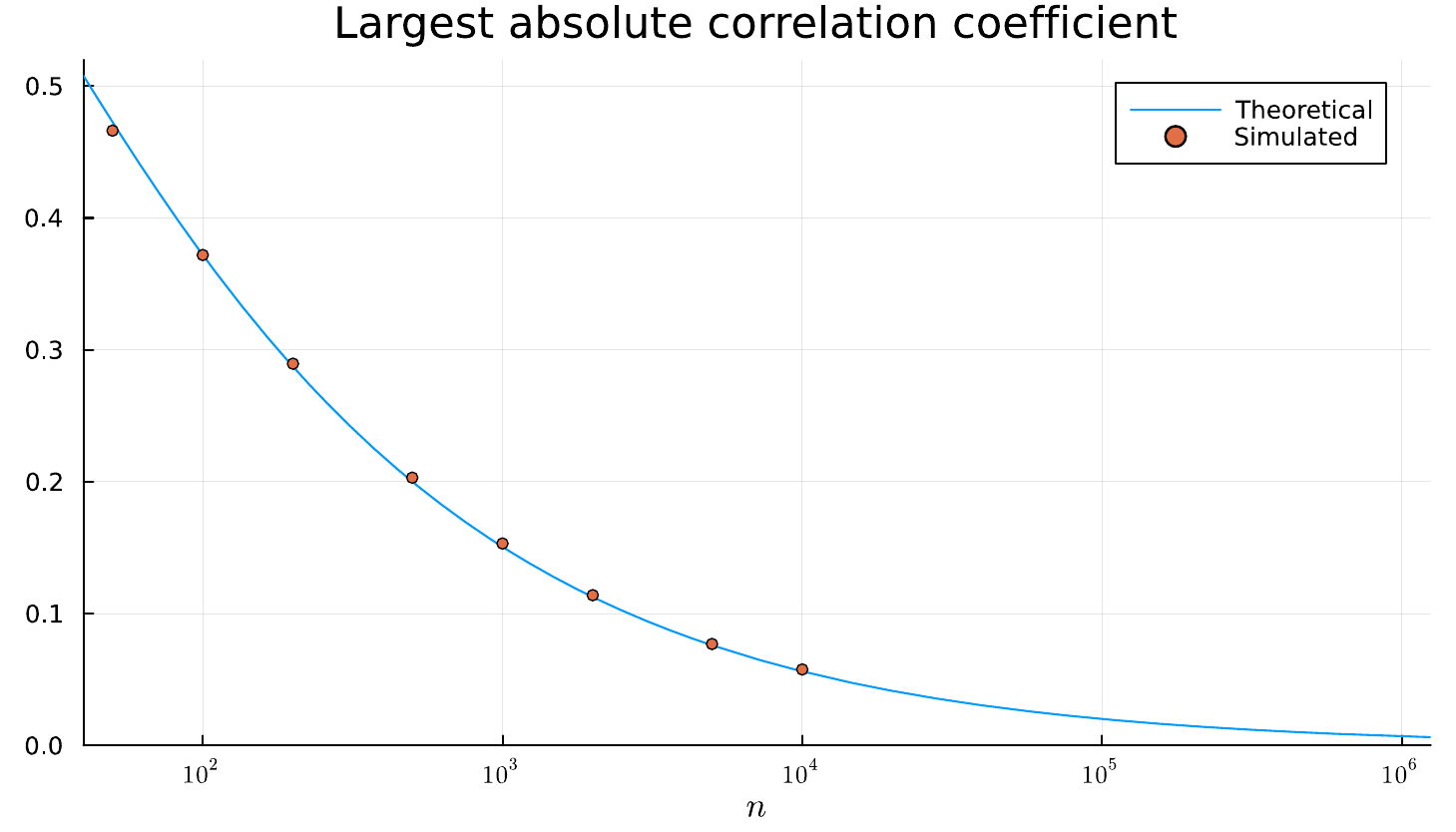}
\caption{{\small{}The largest absolute correlation coefficient in a uniformly distributed $n\times n$ correlation matrix. The theoretical approximation (\ref{eq:maxrhoapprox}) is the solid curve; circles show simulated averages of $M_n$ over independent draws from the uniform distribution on $\mathcal{E}_n$, with the number of replications decreasing from $10{,}000$ at $n=50$ to $50$ draws for $n=10{,}000$. Fewer replications are needed in high dimensions because $M_n$ concentrates, cf.\ Corollary~\ref{cor:gumbel}.\label{fig:CorrNearIdentity}}}
\end{figure}

The approximation (\ref{eq:maxrhoapprox}) is a heuristic typical-value calculation; the rigorous fluctuation limit is the Gumbel law of Corollary~\ref{cor:gumbel} in Section~\ref{sec:extremes}.

Entrywise concentration does not, however, imply global concentration. The next result measures the typical distance from $I_n$ in the Frobenius norm.

\begin{proposition} \label{prop:C-I}
Let $C$ be drawn uniformly from $\mathcal{E}_n$. Then
$$ \mathbb{E}\|C - I_n\|_F^2 = \frac{n(n-1)}{n+1}=n(1+o(1))\qquad\text{and}\qquad \frac{\|C-I_n\|_F^2}{\mathbb{E}\|C-I_n\|_F^2}\overset{p}{\rightarrow}1. $$
\end{proposition}
Proposition~\ref{prop:entrywise} and Proposition~\ref{prop:C-I} together characterize the typical correlation matrix at two different scales. Entrywise, it is nearly indistinguishable from the identity, since by Corollary~\ref{cor:gumbel} even its single largest correlation is only of order $\sqrt{\log n/n}$. Globally, it is far from the identity, with each of the $d=n(n-1)/2$ off-diagonal pairs contributing $2\mathbb{E}[C_{ij}^{2}]=\tfrac{2}{n+1}$ to the squared Frobenius distance, and while every individual term is negligible, their sum, $2d/(n+1)\sim n$, is not. In summary,
$$
\max_{i<j}|C_{ij}| \approx 2\sqrt{\tfrac{\log n}{n}}\rightarrow 0,
\qquad\text{while}\qquad
\|C-I_n\|_F \approx \sqrt{n}\rightarrow\infty.
$$
A typical high-dimensional correlation matrix is therefore locally near, but globally far from, the identity matrix.

\subsection{Spectral Structure and the Sample-Correlation Representation}\label{sec:spectral}

A third measurement looks at the spectrum directly. The following identity is exact at every $n$.

\begin{theorem}[Spectral floor]\label{thm:specfloor}
For $n\geq2$ and $x\in[0,1]$,
\begin{align*}
\{C\in\mathcal{E}_n:\lambda_{\min}(C)\geq x\}&=xI_n+(1-x)\mathcal{E}_n,\\
\operatorname{Vol}\{C\in\mathcal{E}_n:\lambda_{\min}(C)\geq x\}&=(1-x)^{d}\operatorname{Vol}(\mathcal{E}_n).
\end{align*}
\end{theorem}

\noindent{}{\bf Proof.}
Fix $x\in[0,1)$. If $C\in\mathcal{E}_n$ satisfies $\lambda_{\min}(C)\geq x$, then $B=(C-xI_n)/(1-x)$ is symmetric with unit diagonal and $B\succeq0$, so $B\in\mathcal{E}_n$ and $C=xI_n+(1-x)B$. Conversely, if $B\in\mathcal{E}_n$, then $C=xI_n+(1-x)B$ is symmetric with unit diagonal and $C-xI_n=(1-x)B\succeq0$. Hence
$$
\{C\in\mathcal{E}_n:\lambda_{\min}(C)\geq x\}=xI_n+(1-x)\mathcal{E}_n.
$$
In the off-diagonal coordinates $(C_{ij})_{i<j}\in\mathbb{R}^d$, the map $B\mapsto xI_n+(1-x)B$ acts as $C_{ij}=(1-x)B_{ij}$, a linear map with Jacobian determinant $(1-x)^d$, and the volume identity follows from the change of variables for Lebesgue measure. At $x=1$ the set is $\{I_n\}$, which has zero $d$-dimensional volume for $n\geq2$, in agreement with $(1-x)^d=0$.\hfill$\square$

Under the uniform law, $\Pr(\lambda_{\min}(C)\geq x)=(1-x)^d$; equivalently, $\lambda_{\min}(C)\sim\operatorname{Beta}(1,d)$. This exact finite-dimensional law is, to our knowledge, new for the elliptope. The underlying homothety is elementary, and has the following consequences. First, the two constraints live on different exponential scales: the entrywise ceiling $\max_{i<j}|C_{ij}|\leq r$ excludes only a fraction $\exp\{-\tfrac{n}{2}|\log(1-r^2)|+O(\log n)\}$ of the volume (Proposition~\ref{prop:entrywise}), whereas the spectral floor $\lambda_{\min}(C)\geq x$ is highly restrictive: it is satisfied by only the fraction $(1-x)^d$, which decays at rate $n^2$ in the exponent. Second, $d\lambda_{\min}(C)$ converges in distribution to a unit exponential, which recovers the hard-edge scale by elementary means (Remark~\ref{rem:lambdamin}). The probability identity is special to the uniform law: for $\eta\neq1$ the determinant tilt breaks the affine invariance, while the set homothety and the volume identity are geometric and hold regardless.

\begin{corollary}[Projection with a spectral floor]\label{cor:floorprojection}
Fix $x\in[0,1)$, let $A$ be a symmetric $n\times n$ matrix, and let $R^{*}$ be the nearest correlation matrix to $(A-xI_n)/(1-x)$ in the Frobenius norm. Then the nearest matrix to $A$ in $\{C\in\mathcal{E}_n:\lambda_{\min}(C)\geq x\}$ is $xI_n+(1-x)R^{*}$, and
$$
\min_{C\in\mathcal{E}_n:\lambda_{\min}(C)\geq x}\|A-C\|_F=(1-x)\min_{R\in\mathcal{E}_n}\Bigl\|\frac{A-xI_n}{1-x}-R\Bigr\|_F.
$$
\end{corollary}

For a matrix with unit diagonal the recipe is simple: divide the off-diagonal entries by $1-x$, project onto $\mathcal{E}_n$, and map back through $R\mapsto xI_n+(1-x)R$. A spectral floor therefore requires no new projection method: in the Frobenius norm, the constrained problem is the ordinary nearest-correlation problem \citep{Higham:2002} in rescaled coordinates. The constrained problem is well studied: \citet{HighamStrabic:2016} compute nearest correlation matrices subject to $\lambda_{\min}(C)\geq x$ by flooring eigenvalues at $x$ instead of zero within alternating projections. Corollary~\ref{cor:floorprojection} explains the variant geometrically, and any method for the ordinary problem, including the Newton method of \citet{QiSun:2006}, can be reused after rescaling. The reduction also separates projection from shrinkage: direct shrinkage toward the identity, $A\mapsto\omega I_n+(1-\omega)A$ \citep{HighamStrabicSego:2016}, can enforce a spectral floor but need not equal the minimum-distance projection, which the corollary delivers at the cost of a single ordinary nearest-correlation computation.

The link between the elliptope and the Marchenko--Pastur law is an exact distributional identity.

\begin{proposition}[Wishart and Gram representations]\label{prop:WishartGram}
(i) Let $\nu>n-1$ be real, let $W$ have the Wishart distribution $W_n(\nu,I_n)$, and let $C=D^{-1}WD^{-1}$ where $D=\operatorname{diag}(\sigma_1,\ldots,\sigma_n)$ with $\sigma_i^2=W_{ii}$. Then $C$ has density proportional to $\det(C)^{(\nu-n-1)/2}$ on $\mathcal{E}_n$; equivalently, $C\sim\operatorname{LKJ}(\eta)$ with $\nu=n+2\eta-1$, for every $\eta>0$.\newline
(ii) For integer $\nu=k\geq n$, $C$ is distributed as the Gram matrix $C=V^\prime V$, $V=(v_1,\ldots,v_n)$, of $n$ independent vectors $v_i$ that are uniformly distributed on the unit sphere in $\mathbb{R}^k$. In particular, $C$ is uniformly distributed on $\mathcal{E}_n$ if and only if $k=n+1$.
\end{proposition}

Part (i) is classical for integer $\nu$, as the null density of the sample correlation matrix \citep[theorem~5.1.3]{Muirhead:1982}; its identification with the LKJ family of \citet{LewandowskiKurowickaJoe:2009} is due to \citet{WangWuChu:2018}. Part (ii) is the sample-correlation representation. Part (i) also gives a direct Bartlett sampler for $\operatorname{LKJ}(\eta)$ correlation matrices: draw $W\sim W_n(n+2\eta-1,I_n)$ by the Bartlett decomposition and normalize, the restricted-Wishart sampler of \citet{WangWuChu:2018}; row-normalizing the Bartlett factor itself yields a Cholesky factor of the LKJ draw directly. See \citet{Hansen:BartlettLKJ} for implementation details, exact couplings with the samplers of \citet{LewandowskiKurowickaJoe:2009}, and benchmarks.

At $\eta=1$ the normalizing constant of the LKJ density is $\operatorname{Vol}(\mathcal{E}_n)$ itself, and the product formula~(\ref{eq:VolCn}) is the ratio of the Wishart and chi-squared normalizing constants in Proposition~\ref{prop:WishartGram}(i) at $\nu=n+1$ (Section~\ref{sec:barnes} of the Supplement).

Writing $v_i=z_i/\|z_i\|$ with $z_1,\ldots,z_n$ i.i.d.\ $N(0,I_k)$, the entries $C_{ij}=z_i^\prime z_j/(\|z_i\|\|z_j\|)$ are the (uncentered) sample correlations of $n$ independent Gaussian variables computed from $T=k$ observations. So, Proposition~\ref{prop:WishartGram} shows that a uniform draw from $\mathcal{E}_n$ is distributed exactly as the empirical correlation matrix of $n$ independent Gaussian variables observed $T=n+1$ times. Only the spherical symmetry of the Gaussian distribution is used here, so the identity extends to any $z_1,\ldots,z_n$ that are independent across variables and spherically distributed in $\mathbb{R}^{T}$, but not to elliptically distributed observation rows. The same holds for the centered (Pearson) sample correlation matrix based on $T=n+2$ observations. Centering projects each column onto the $(n+1)$-dimensional subspace orthogonal to the vector of ones, and after normalization the centered columns are independent and uniform on the unit sphere in that subspace. In either form the uniform law is the law of a sample correlation matrix with $n/T\rightarrow1$, so the limiting theory of sample correlation matrices applies.

\begin{theorem}[Marchenko--Pastur law]\label{thm:MP}
Let $C_n$ be uniformly distributed on $\mathcal{E}_n$, with eigenvalues $\lambda_1,\ldots,\lambda_n$, and let $F_n(x)=\frac{1}{n}\#\{i:\lambda_i\leq x\}$ denote the empirical spectral distribution. Then, as $n\rightarrow\infty$, $F_n$ converges weakly, in probability, to the Marchenko--Pastur law with ratio one, whose density is
$$
f_{\mathrm{MP}}(x)=\frac{1}{2\pi x}\sqrt{x(4-x)},\qquad x\in(0,4].
$$
\end{theorem}

\begin{remark}[Coupling]\label{rem:coupling}
Almost-sure convergence holds when the $C_n$ are coupled on a common probability space as follows. Let $(z_{ti})_{t\geq1,i\geq1}$ be an infinite array of independent $N(0,1)$ entries. For each $n$, let $z_i^{(n)}=(z_{1i},\ldots,z_{n+1,i})^\prime\in\mathbb{R}^{n+1}$ and set $v_i^{(n)}=z_i^{(n)}/\|z_i^{(n)}\|$; then $C_n=(v_i^{(n)\prime}v_j^{(n)})_{i,j=1}^n$ is uniformly distributed on $\mathcal{E}_n$ by Proposition~\ref{prop:WishartGram}(ii). Since $C_n$ is the uncentered sample correlation matrix of the first $n$ variables observed $T=n+1$ times, almost-sure weak convergence of $F_n$ follows from \citet[theorem~2.2]{Heiny:2022}.
\end{remark}

Figure~\ref{fig:MPlaw} illustrates the convergence by comparing empirical spectral distributions of uniform draws from $\mathcal{E}_n$ with the Marchenko--Pastur density.

The limiting spectrum reproduces the two scales established above. The second moment of the Marchenko--Pastur law about one is $\int(x-1)^2f_{\mathrm{MP}}(x)dx=1$, so $\|C-I_n\|_F^2=\sum_i(\lambda_i-1)^2\approx n$, which is shown in Proposition~\ref{prop:C-I}. The eigenvalues spread over the entire interval $[0,4]$, with substantial mass near zero, so a uniform draw becomes asymptotically ill-conditioned: the typical matrix sits close to the singular boundary of $\mathcal{E}_n$, even though every individual correlation is tiny. Section~\ref{sec:precision} quantifies this ill-conditioning.

\begin{figure}[htbp!]
\centering{}\includegraphics[width=0.8\textwidth]{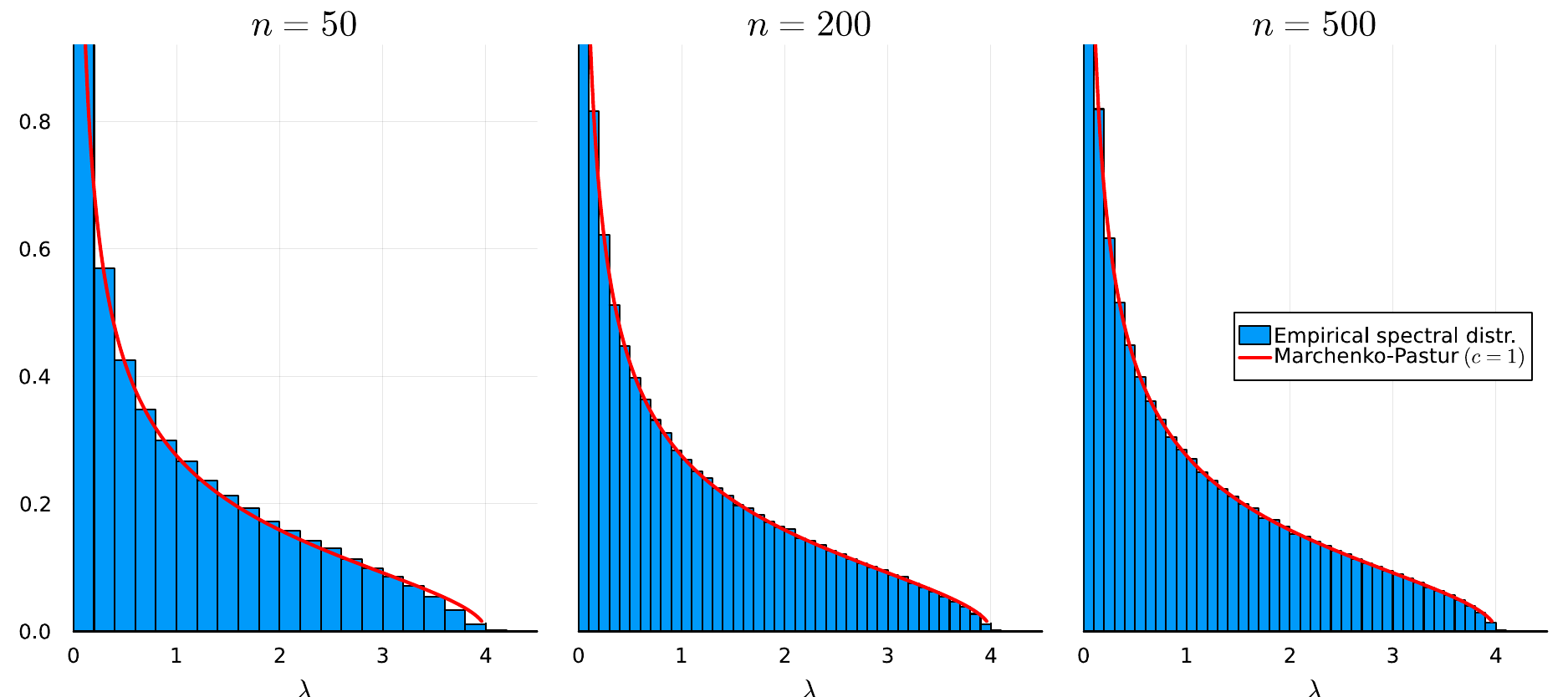}
\caption{{\small{}Empirical spectral distributions of uniformly distributed $n\times n$ correlation matrices for $n\in\{50,200,500\}$, based on pooled eigenvalues from independent draws from $\mathcal{E}_n$ via the LKJ$(1)$ representation. Each histogram pools $1{,}000$ eigenvalues: $20$ draws at $n=50$, $5$ draws at $n=200$, and $2$ draws at $n=500$. The Marchenko--Pastur density with ratio one, $f_{\mathrm{MP}}(x)=(2\pi x)^{-1}\sqrt{x(4-x)}$, is shown in red. Convergence to the limiting law is visually apparent already at moderate $n$.\label{fig:MPlaw}}}
\end{figure}

The Marchenko--Pastur law describes the bulk of the spectrum; the following remark quantifies the smallest eigenvalue.

\begin{remark}[Hard edge]\label{rem:lambdamin}
The Marchenko--Pastur law places the left edge of its support at zero when the ratio is one, so $\lambda_{\min}(C_n)\to0$ in probability (almost surely under the coupling of Remark~\ref{rem:coupling}). The rate and the limit law follow exactly from Theorem~\ref{thm:specfloor}: since $\lambda_{\min}(C_n)\sim\operatorname{Beta}(1,d)$,
$$
\Pr\bigl(d\lambda_{\min}(C_n)>t\bigr)=\bigl(1-t/d\bigr)^{d}\rightarrow e^{-t},
\qquad\text{i.e.}\qquad
d\lambda_{\min}(C_n)\overset{d}{\rightarrow}\operatorname{Exp}(1),
$$
so $\lambda_{\min}(C_n)=O_p(n^{-2})$, far smaller than the $O(1)$ bulk eigenvalue scale. This is consistent with the classical hard-edge theory: by the sample-correlation representation with $T=n+1$, the uniform law sits at the hard edge of the real Laguerre family, where the $\beta=1$ Laguerre hard-edge theory \citep[section~7.9]{Forrester:2010} gives $4n\lambda_{\min}(W)\overset{d}{\rightarrow}\xi_{\min}$ for $W\sim W_n(n+1,I_n)$, with $\xi_{\min}$ the smallest point of the corresponding $\beta=1$ Bessel process. Passing to $C_n$ by the diagonal-normalization argument of Step~2 in the proof of Theorem~\ref{thm:LKJscaling} yields $d\lambda_{\min}(C_n)\overset{d}{\rightarrow}\tfrac18\xi_{\min}$, and matching the two limits shows that $\xi_{\min}$ is exponential with mean eight. The exponential law itself is classical: at $\nu=n+1$ the smallest Wishart eigenvalue satisfies $\Pr(\lambda_{\min}(W)>x)=e^{-nx/2}$ exactly for every $n$, a case of the finite-$n$ results of \citet[section~10]{Edelman:1988}; Theorem~\ref{thm:specfloor} gives an independent derivation on the correlation side.
\end{remark}

\subsection{Extreme Correlations and Poisson Limits}\label{sec:extremes}

We now characterize the full collection of extreme off-diagonal entries. The key input is exact pairwise independence, which holds throughout the LKJ family.

\begin{lemma}[Pairwise independence]\label{lem:pairwise}
Let $C\sim\operatorname{LKJ}_n(\eta)$ with $\eta>0$. Any two distinct off-diagonal entries $C_{ij}$ and $C_{kl}$, $\{i,j\}\neq\{k,l\}$, are independent. In particular this holds for the uniform distribution on $\mathcal{E}_n$ ($\eta=1$).
\end{lemma}

Lemma~\ref{lem:pairwise} extends to mutual independence along any forest. In the regular-vine construction of \citet{LewandowskiKurowickaJoe:2009}, the first-tree entries are independent Beta variables, and any forest extends to a spanning tree that can serve as the first tree. Dependence among the raw correlations therefore first appears on cycles, the smallest of which is a triangle.

\begin{lemma}[Beta-tail intensity]\label{lem:betatail}
Let $C\sim\operatorname{LKJ}_n(\eta_n)$ with $\eta_n>0$, let $\nu_n=n+2\eta_n-1$ and $\mu_n=\log\tfrac{n^4}{\log n}$, and let $\Lambda$ be the measure on $\mathbb{R}$ with
$$
\Lambda(dt)=\tfrac{1}{2\sqrt{8\pi}}e^{-t/2}dt,\qquad\text{so that }\Lambda(t,\infty)=\tfrac{1}{\sqrt{8\pi}}e^{-t/2}.
$$
Then
$$
d\cdot\Pr\{\nu_nC_{12}^2-\mu_n>t\}\rightarrow\Lambda(t,\infty),\qquad n\to\infty,
$$
for every $t\in\mathbb{R}$. More generally, $d\cdot\Pr\{\nu_nC_{12}^2-\mu_n\in B\}\to\Lambda(B)$ for every Borel set $B\subseteq\mathbb{R}$ that is bounded away from $-\infty$ with $\Lambda(\partial B)=0$.
\end{lemma}

Thresholding $C$ at level $\tau\in(0,1)$ produces a correlation network with edge count $N_n(\tau)$ and edge probability $p_n(\tau)$, defined as in Section~\ref{sec:twoscale} but now under $\operatorname{LKJ}_n(\eta_n)$.
The exceedance point process associated with $C$ is the random counting measure defined, for a Borel set $B\subseteq\mathbb{R}$, by
$$
\Xi_n(B)=\#\{i<j:\nu_nC_{ij}^2-\mu_n\in B\};
$$
in particular, $\Xi_n(t,\infty)$ counts how many of the $d$ normalized squared correlations exceed $t$. For a Borel set $B$ we write $\lambda_n(B)=\mathbb{E}[\Xi_n(B)]=d\cdot\Pr\{\nu_nC_{12}^2-\mu_n\in B\}$ for the \emph{pre-limit} mean. Lemma~\ref{lem:betatail} gives $\lambda_n(t,\infty)\to\Lambda(t,\infty)$, so the expected count converges to the intensity of a Poisson process with measure $\Lambda(dt)=\frac{1}{2\sqrt{8\pi}}e^{-t/2}dt$. Theorem~\ref{thm:PPP} below shows that $\Xi_n$ converges to that Poisson process in distribution; formally, convergence is in the vague topology on the space of Radon measures on $(-\infty,\infty]$.

\begin{theorem}[Poisson point process]\label{thm:PPP}
Let $C\sim\operatorname{LKJ}_n(\eta_n)$ with $\eta_n>0$; the uniform law is the case $\eta_n=1$. The exceedance process $\Xi_n$ converges to a Poisson point process with intensity $\Lambda(dt)=\frac{1}{2\sqrt{8\pi}}e^{-t/2}dt$. Equivalently, for any finite collection of disjoint Borel sets $B_1,\ldots,B_m$ bounded away from $-\infty$ with $\Lambda(\partial B_r)=0$, the counts $\bigl(\Xi_n(B_1),\ldots,\Xi_n(B_m)\bigr)$ converge in distribution to independent Poisson variables with means $\Lambda(B_1),\ldots,\Lambda(B_m)$.
\end{theorem}

For the uniform law this theorem is not new as a statement: by the sample-correlation representation, point-process convergence of the off-diagonal entries is covered by the general theory of \citet[theorem~3.7]{HeinyMikoschYslas:2021}. The contribution here is the proof, which covers every $\operatorname{LKJ}(\eta_n)$ law, including noninteger $\nu_n$, where no sample-correlation representation is available: Lemma~\ref{lem:pairwise} gives pairwise independence of off-diagonal entries exactly (not just asymptotically), so the joint-expectation terms in the Chen--Stein bound reduce to products of marginal probabilities, and the two nonzero error terms in the bound of \citet{ArratiaGoldsteinGordon:1989} are of order $n^{-1}$. The argument is simple and delivers an explicit $O(n^{-1})$ total-variation bound for the finite-dimensional exceedance counts relative to Poisson laws with the exact means $\lambda_n(B_r)$. No rate is claimed for the limit process itself, since that step uses only the convergence $\lambda_n(B_r)\rightarrow\Lambda(B_r)$. Neither $\mu_n$ nor $\Lambda$ depends on $\eta_n$, because the tail argument uses only $(\log n)^2/\nu_n\rightarrow0$, which holds since $\nu_n>n-1$.

\begin{corollary}[Edge counts and triangles]\label{cor:ppp_consequences}
Let $C\sim\operatorname{LKJ}_n(\eta_n)$ with $\eta_n>0$. (a) If the thresholds $\tau_n\in(0,1)$ satisfy $\mathbb{E}N_n(\tau_n)\rightarrow\lambda\in(0,\infty)$, which holds for $\tau_n^2=(\mu_n-\log(8\pi\lambda^2)+o(1))/\nu_n$, then $N_n(\tau_n)\overset{d}{\rightarrow}\operatorname{Poisson}(\lambda)$; and (b) at these thresholds, the probability that the thresholded graph contains a triangle tends to zero.
\end{corollary}

\begin{corollary}[Gumbel limit]\label{cor:gumbel}
Let $C\sim\operatorname{LKJ}_n(\eta_n)$ with $\eta_n>0$, and let $M_n=\max_{1\leq i<j\leq n}|C_{ij}|$. Then, as $n\rightarrow\infty$,
$$
\nu_nM_n^2-\mu_n\overset{d}{\rightarrow}\mathcal{G},\qquad \mathcal{G}(t)=\exp\{-\tfrac{1}{\sqrt{8\pi}}e^{-t/2}\},
$$
the Gumbel distribution with location $-\log(8\pi)$ and scale $2$. In particular, $\sqrt{\nu_n/\log n}M_n\overset{p}{\rightarrow}2$.
\end{corollary}

The rescaled maximum $\nu_nM_n^2-\mu_n$ is the largest atom of $\Xi_n$, so the Gumbel limit follows from Theorem~\ref{thm:PPP} and $\Pr\{\Xi_n(t,\infty)=0\}\rightarrow\exp\{-\Lambda(t,\infty)\}$. For the uniform law, via the sample-correlation representation, Corollary~\ref{cor:gumbel} is the specialization to $T=n+2$ of the classical largest-entry result of \citet{Jiang:2004maxentry}; it is the rigorous counterpart to the typical-value approximation (\ref{eq:maxrhoapprox}), which sets $\nu_nM_n^2-\mu_n$ equal to the mode $-\log(8\pi)$ of $\mathcal{G}$.

Under the uniform law, the edge indicators of the thresholded network have common probability $p_n(\tau)$ and are pairwise independent by Lemma~\ref{lem:pairwise}. Consequently, every edge count has exactly the same mean and variance as under the Erd\H{o}s--R\'enyi model $G(n,p_n)$, in which edges are mutually independent. That assumption is inconsistent with the thresholded network, whose edges are only pairwise independent because they all come from one positive semidefinite matrix. At extreme thresholds, where $\mathbb{E}N_n(\tau_n)\rightarrow\lambda$, edge counts converge to a Poisson distribution and triangles disappear with probability tending to one (Corollary~\ref{cor:ppp_consequences}), so the dependence is negligible to first order. Figure~\ref{fig:PPP} in the Supplement illustrates the convergence of $\Xi_n$ to its limit. The uniform law therefore provides a null in which every edge comes from one positive semidefinite matrix and whose sparse extreme limit agrees with the Erd\H{o}s--R\'enyi benchmark. At the moderate thresholds used in practice the dependence is visible in triangle counts, and it is computable exactly: every $3\times3$ principal submatrix of a uniform draw is $\operatorname{LKJ}_3(\tfrac{n-1}{2})$, and at edge density $p=p_n(\tau)$ the expected number of triangles exceeds the Erd\H{o}s--R\'enyi value by the factor $1+\tfrac{1}{2(n+1)}(2z\phi(z)/p)^3+o(n^{-1})$, where $z=\Phi^{-1}(1-p/2)$ and $\phi$ is the standard normal density (Section~\ref{sec:triangles} of the Supplement): $1.4$ at $n=200$ and $p=0.03$, and $3.0$ at $n=100$ and $p=0.01$. The inflation is a finite-$n$ effect that varies with $n$ and density, so it matters most for comparisons across periods or sample sizes. For Pearson correlations from $T$ Gaussian observations, the uniform law is the exact null only at $T=n+2$. At other sample sizes the same computation applies with $n+1$ replaced by the spherical dimension $T-1$, while residualization with estimated factors introduces further dependence, so empirical calibration should use the appropriate sampling ensemble.

Correlation-aware null models exist: \citet{MacMahonGarlaschelli:2015} filter the spectrum with random-matrix bounds for community detection, \citet{KojakuMasuda:2019} select a null model by regularized likelihood, and \citet{MasudaBoydGarlaschelliMucha:2025} note the triangle inflation qualitatively; the present null is exact at every $n$, with closed-form $p_n(\tau)$ and the bound of Theorem~\ref{thm:PPP}. The setting is close in spirit to the correlation-screening framework of \citet{HeroRajaratnam:2011}, extended to compound-Poisson characterizations in ultra-high dimension by \citet{WeiRajaratnamHero:2023}.

\section{Implications for High-Dimensional Modeling}\label{sec:implications}

\subsection{Sampling Noise in Empirical Correlation Matrices}\label{sec:MP}

Let $\hat{C}$ be the sample correlation matrix computed from $T$ independent observations of $n$ independent standardized Gaussian variables, the setting in which the Wishart identity is exact. Its sampling noise relates to the elliptope in three ways. First, $\operatorname{var}(\hat{C}_{ij})=1/(T-1)$ exactly for the Pearson correlation matrix, and hence
$$
\mathbb{E}\|\hat{C}-I_n\|_F^2=\frac{n(n-1)}{T-1},
$$
so when $T\asymp n$ the accumulated entrywise noise produces a Frobenius displacement of order $\sqrt{n}$: the same global scale as a uniform draw from the elliptope (Proposition~\ref{prop:C-I}). (For non-Gaussian data the $1/T$ variance scale and the spectral limit below continue to hold under standard moment conditions, e.g.\ independent standardized entries with finite fourth moments \citep{Jiang:2004,BaiSilverstein:2010}, but the exact identities of Proposition~\ref{prop:WishartGram} require spherical, in particular Gaussian, observations.)

Second, by the Wishart/LKJ identity, the limiting spectral distribution for $n/T\rightarrow c\in(0,1]$ is the Marchenko--Pastur law with ratio $c$. Its left edge $(1-\sqrt{c})^2$ reaches zero at $c=1$, so as $c\uparrow1$ the eigenvalues spread out and accumulate near the singular boundary.

Third, at $T=n+2$ the law of $\hat{C}$ is exactly the uniform distribution on $\mathcal{E}_n$ (Proposition~\ref{prop:WishartGram} and the centered representation that follows it). For other sample sizes $T>n$ it is the LKJ law with $\eta=(T-n)/2$, which tilts the uniform law by $\det(C)^{\eta-1}$. For the uncentered matrix the corresponding values are $T=n+1$ and $\eta=(T-n+1)/2$, for $T\geq n$. At any other sample size, a sample correlation matrix at the Frobenius scale of the elliptope is therefore not spread uniformly over it.

\subsection{Concentration of Correlation Priors in High Dimensions}\label{sec:LKJ}

By Proposition~\ref{prop:WishartGram}(i), the LKJ$(\eta)$ law of \citet{LewandowskiKurowickaJoe:2009}, with density $\pi_\eta(C)\propto\det(C)^{\eta-1}$ on $\mathcal{E}_n$, is the law of the correlation matrix of a Wishart with $\nu=n+2\eta-1$ degrees of freedom. At $\eta=1$ it reduces to the uniform distribution studied in Section~\ref{sec:concentration}, the jointly uniform prior of \citet{BarnardMcCullochMeng:2000}. The limit theory of sample correlation matrices therefore applies to every LKJ law, with aspect ratio $c_n=n/\nu_n$. The following result characterizes the LKJ law when $\eta$ is allowed to grow with the dimension.

\FloatBarrier
\begin{theorem}[LKJ scaling]\label{thm:LKJscaling}
Let $C_n\sim\operatorname{LKJ}_n(\eta_n)$ with $\eta_n>0$, and let $\nu_n=n+2\eta_n-1$.\newline
(i) If $\eta_n/n\rightarrow\theta\in[0,\infty)$, the empirical spectral distribution of $C_n$ converges weakly, in probability, to the Marchenko--Pastur law with ratio $c_\theta=1/(1+2\theta)$.\newline
(ii) The Frobenius distance from the identity satisfies
$$
\mathbb{E}\|C_n-I_n\|_F^2=\frac{n(n-1)}{\nu_n}\qquad\text{and}\qquad\frac{\|C_n-I_n\|_F^2}{\mathbb{E}\|C_n-I_n\|_F^2}\overset{p}{\rightarrow}1.
$$
In particular, $\tfrac{1}{n}\|C_n-I_n\|_F^2\rightarrow1/(1+2\theta)$ in probability when $\eta_n/n\rightarrow\theta$, and when $\eta_n/n\rightarrow\infty$ the Frobenius distance diverges if $\eta_n=o(n^2)$, is bounded in probability if $\eta_n\asymp n^2$, and vanishes if $n^2=o(\eta_n)$.
\end{theorem}

Theorem~\ref{thm:MP} is the case $\eta_n=1$. By Proposition~\ref{prop:WishartGram}, $\nu_n=n+2\eta_n-1$ is the Wishart degrees of freedom, a sample size when it is an integer, so raising $\eta$ is analogous to adding observations. The spectrum changes only when the number of observations is proportional to $n$, the familiar aspect-ratio condition for sample correlation matrices, and the expected squared Frobenius distance from the identity, which is the sum of the $n(n-1)$ entry variances $1/\nu_n$, stays bounded only when the number of observations is of order $n^2$. Figure~\ref{fig:LKJphase} illustrates the two scales.

\begin{figure}[!htbp]
\centering
\includegraphics[width=0.8\textwidth]{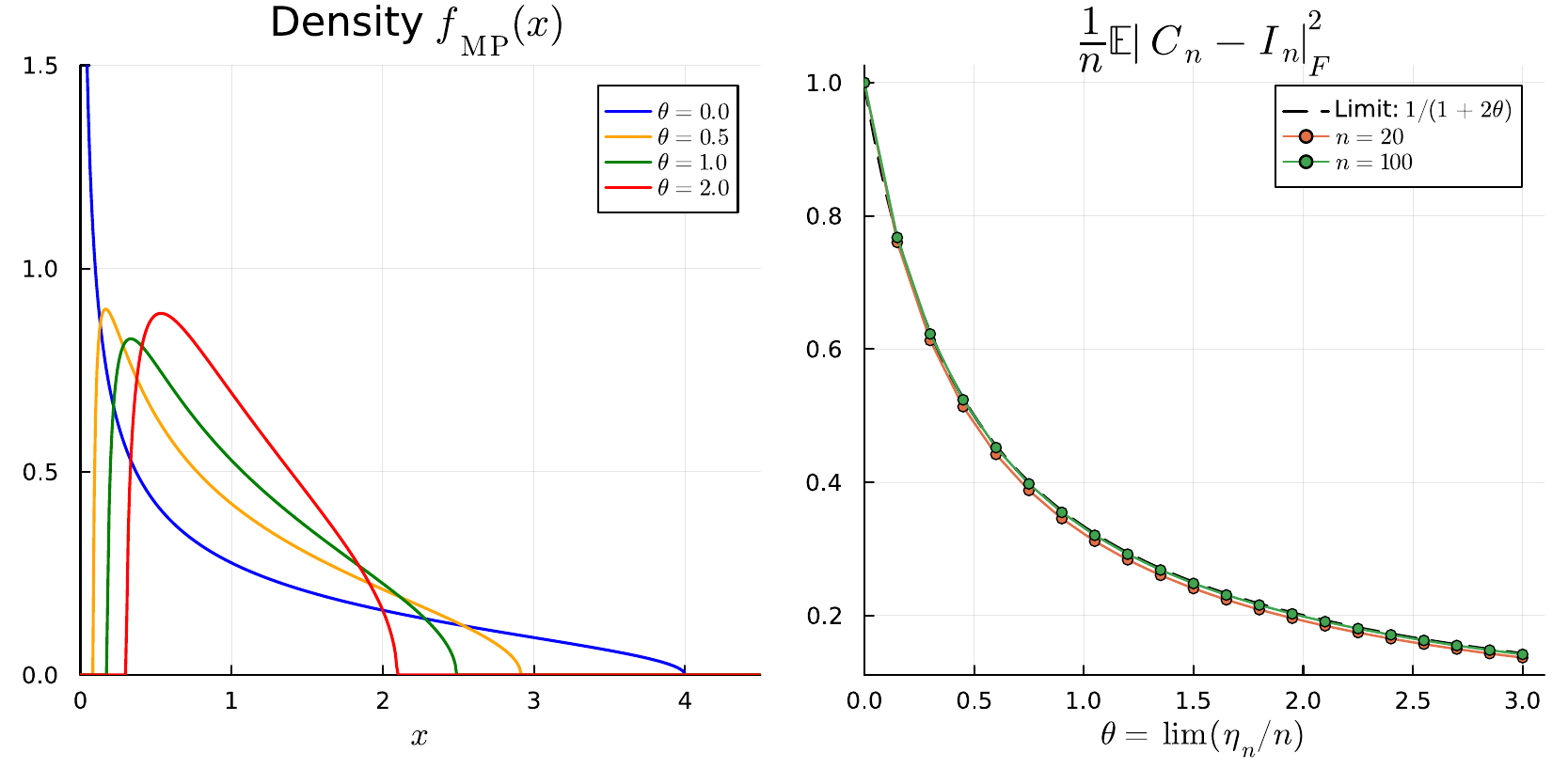}
\caption{Dimension-dependent LKJ scaling. The left panel shows the Marchenko--Pastur limits from Theorem~\ref{thm:LKJscaling} when $\eta_n/n\rightarrow\theta$, so that the limiting aspect ratio is $c_\theta=1/(1+2\theta)$. The right panel shows the \emph{normalized} Frobenius scale under linear scaling $\eta_n=\theta n$, namely $\tfrac{1}{n}\mathbb{E}\|C_n-I_n\|_F^2=(n-1)/(n+2\eta_n-1)$, evaluated at $n=20$ and $n=100$, together with its limit $1/(1+2\theta)$ from Theorem~\ref{thm:LKJscaling}(ii). Linear scaling thus changes the limiting spectrum and the normalized Frobenius scale, but leaves $\|C_n-I_n\|_F^2$ of order $n$; scaling at least of order $n^2$ is needed to keep the unnormalized Frobenius distance from the identity bounded.}
\label{fig:LKJphase}
\end{figure}

The central fact about fixed hyperparameters is that they are all alike to first order. Under $\operatorname{LKJ}(\eta)$ the marginal variance of an off-diagonal entry is exactly
\begin{equation}\label{eq:LKJvar}
\operatorname{var}(C_{ij})=\frac{1}{n+2\eta-1}=\frac{1}{\nu_n}\sim\frac{1}{n}\qquad\text{for every fixed }\eta>0,
\end{equation}
and by Theorem~\ref{thm:LKJscaling}(i) with $\theta=0$ the limiting spectrum is the Marchenko--Pastur law with ratio one, again for every fixed $\eta>0$. The entrywise concentration near $I_n$ is therefore not produced by the tilt $\det(C)^{\eta-1}$. It is a property of the volume measure on the elliptope, which every tilt inherits, and it holds for $0<\eta<1$ as well. A fixed $\eta$ changes only the constant in (\ref{eq:LKJvar}) and, more visibly, the determinant and the conditioning of $C$, as we now describe.

For any correlation matrix $C$, $0\leq\det(C)\leq1$, with $\det(C)=1$ if and only if $C=I_n$. \citet{Joe:2006} shows that $\det(C)=\prod_{i<j}(1-\varrho_{ij}^2)$, where the $\varrho_{ij}=\rho_{i,j;i+1,\ldots,j-1}$ are the $d$ partial correlations of a D-vine, so each nonzero partial correlation lowers the determinant and moves the matrix closer to the singular boundary, where $\det(C)=0$. The size of this effect under the uniform law follows from Proposition~\ref{prop:WishartGram}(i). With $C=D^{-1}WD^{-1}$, $W\sim W_n(n+1,I_n)$ and $D=\operatorname{diag}(W_{11}^{1/2},\ldots,W_{nn}^{1/2})$, the Bartlett decomposition gives $\det(W)\overset{d}{=}\prod_{i=1}^{n}\chi^2_{n+2-i}$ while $W_{ii}\sim\chi^2_{n+1}$, and $\mathbb{E}[\log\chi^2_m]=\psi(\tfrac{m}{2})+\log2$ then yields
$$
\mathbb{E}[\log\det(C)]=\sum_{j=2}^{n+1}\psi(\tfrac{j}{2})-n\psi(\tfrac{n+1}{2})=-n+\tfrac{1}{2}\log(2n)+\tfrac{1+\gamma}{2}+o(1),
$$
where $\psi$ is the digamma function and $\gamma$ the Euler--Mascheroni constant. Hence $\log\det(C)=-n+o_p(n)$, even under the uniform law, in agreement with Theorem~\ref{thm:MP} and $\int\log x\,f_{\mathrm{MP}}(x)\,dx=-1$. \citet{HaneaNane:2018} study the distribution of $\det(C)$ under the uniform law and prove a central limit theorem for $\log\det(C)$, which Section~\ref{sec:centering} of the Supplement restates with the exact centering above. 

Because the LKJ$(\eta)$ density is proportional to $\det(C)^{\eta-1}$, this exponentially small determinant contributes a factor of order $\exp\{-(\eta-1)n\}$ to the density. For $\eta>1$ the LKJ law therefore places exponentially more mass on matrices with relatively large determinants, and for $0<\eta<1$ the negative exponent runs the same argument in reverse, tilting mass toward smaller determinants and more ill-conditioned matrices. Neither tilt is strong enough to matter at first order. By (\ref{eq:LKJvar}) the marginal variance is $1/\nu_n\sim1/n$ in both cases, by Theorem~\ref{thm:LKJscaling}(i) the limiting spectrum is the Marchenko--Pastur law with ratio one in both cases, and $\mathbb{E}\|C-I_n\|_F^2=n(n-1)/\nu_n\approx n$ diverges in both cases. This is the point of Theorem~\ref{thm:LKJscaling}: $\eta$ must grow with $n$ before it changes either of the two scales of Section~\ref{sec:concentration}. A fixed $\eta$ does, however, decide whether the expected variance inflation factor $\mathbb{E}[(C^{-1})_{jj}]$ is finite, which it is if and only if $\eta>1$ (Section~\ref{sec:precision}).

No choice of $\eta$, fixed or varying with $n$, avoids the entrywise shrinkage. The density $\det(C)^{\eta-1}$ is integrable only for $\eta>0$, so by (\ref{eq:LKJvar}) every LKJ law has $\operatorname{var}(C_{ij})<1/(n-1)$, and preventing the variance from vanishing would require $\eta_n\approx-n/2$, which is outside the LKJ family. The two scales in Theorem~\ref{thm:LKJscaling} determine when the spectral and Frobenius behavior can nevertheless change. The constraint is specific to the LKJ parameterization, not a universal impossibility for Bayesian priors on correlation matrices. For instance, priors built through an unconstrained reparameterization, such as the Generalized Fisher Transformation \citep{ArchakovHansen:Correlation} or partial-correlation sequences \citep{Joe:2006,LewandowskiKurowickaJoe:2009}, can be specified so as to escape it; see \citet{ArchakovHansenLuo-RandomCorr:2024} for an implementation.

\subsection{Precision Matrices and the Edge of Invertibility}\label{sec:precision}

Several applications use the inverse of the correlation matrix rather than the matrix itself. Minimum-variance portfolio weights are proportional to $\Sigma^{-1}\iota$, where $\iota$ is the vector of ones, generalized least squares uses $C^{-1}$, and partial correlations are the negatives of the normalized off-diagonal entries of the precision matrix. The following proposition states a consequence of Theorem~\ref{thm:specfloor} for the uniform law and a known consequence of the Marchenko--Pastur theory \citep{BaiSilverstein:2010,Jiang:2004} for sample correlation matrices.

\begin{proposition}\label{prop:precision}
(i) Let $C_n$ be uniformly distributed on $\mathcal{E}_n$. Then
$$
\tfrac{1}{n}\operatorname{tr}(C_n^{-1})\overset{p}{\rightarrow}\infty\qquad\text{and}\qquad n^2/\operatorname{tr}(C_n^{-1})=O_p(1).
$$
(ii) Let $\hat{C}_n$ be the sample correlation matrix of $n$ independent Gaussian variables based on $T$ observations, where $n/T\rightarrow c\in(0,1)$. Then
$$
\tfrac{1}{n}\operatorname{tr}(\hat{C}_n^{-1})\overset{\mathrm{as}}{\rightarrow}\frac{1}{1-c}.
$$
\end{proposition}

Part (ii) follows from $\tfrac{1}{n}\operatorname{tr}(C^{-1})=\tfrac{1}{n}\sum_i\lambda_i^{-1}\to\int x^{-1}f_{\mathrm{MP},c}(x)\,dx=1/(1-c)$ for $c<1$. Part (i) follows from Theorem~\ref{thm:specfloor}, because $\operatorname{tr}(C_n^{-1})\geq1/\lambda_{\min}(C_n)$ and $\lambda_{\min}(C_n)=O_p(n^{-2})$. The average precision of a uniform draw thus grows at least linearly in $n$, in line with the value $1/(1-c)=n+1$ of the Marchenko--Pastur formula at the aspect ratio $c=n/(n+1)$ of the uniform law. At ratio one the Marchenko--Pastur density behaves like $\pi^{-1}x^{-1/2}$ near zero, which puts enough mass near singularity for the average inverse eigenvalue to diverge.

The quantity $(C^{-1})_{jj}=1/(1-R_{j}^{2})$ is the variance inflation factor of variable $j$, where $R_j^2$ is its coefficient of determination on the other $n-1$ variables, so $\tfrac{1}{n}\operatorname{tr}(C^{-1})$ is the average variance inflation factor. Under the uniform law it is finite almost surely at each $n$ and diverges only in probability, but its expectation is infinite already at every $n\geq2$. Under $\operatorname{LKJ}_n(\eta)$, $R_j^2\sim\operatorname{Beta}(\tfrac{n-1}{2},\eta)$, the classical null law of the squared multiple correlation coefficient for integer degrees of freedom \citep[theorem~5.2.2]{Muirhead:1982}, so $\mathbb{E}[(C^{-1})_{jj}]=1+\tfrac{n-1}{2(\eta-1)}$ for $\eta>1$ and is infinite for $\eta\leq1$ (Section~\ref{sec:auxproofs} of the Supplement). The threshold $\eta=1$ corresponds to $T=n+1$ in the Wishart identity $\mathbb{E}[\hat{\Sigma}^{-1}]=\tfrac{T}{T-n-1}\Sigma^{-1}$ for the uncentered Gaussian covariance estimator, which is finite only for $T>n+1$. The uniform law therefore sits at the boundary where the expected precision stops being finite, and part (ii) describes the approach to that boundary, where the average variance inflation factor tends to $1/(1-c)$, which diverges as $c\rightarrow1$.

Inverse problems are therefore unstable in the critical regime $n/T\rightarrow1$. Shrinkage, factor structure, or regularization of the spectrum, such as the condition-number regularization of \citet{WonLimKimRajaratnam:2013}, is the natural response. A spectral floor is the simplest such regularization. The constraint $\lambda_{\min}(C)\geq x$ gives $\|C^{-1}\|\leq x^{-1}$ and the sharp condition-number bound $\lambda_{\max}(C)/\lambda_{\min}(C)\leq(n-(n-1)x)/x$, and Corollary~\ref{cor:floorprojection} computes the projection with this floor by one ordinary nearest-correlation computation.

\subsection{Entrywise Specification and the Nearest Correlation Matrix}\label{sec:nearest}

Correlation matrices are sometimes specified entrywise rather than estimated jointly. Experts elicit pairwise correlations one at a time, regulators prescribe correlation shocks in stress scenarios, and pairwise estimates are assembled from different samples, time windows, or data sources. The result is a symmetric matrix $\tilde{C}$ with unit diagonal and entries in $[-1,1]$ that need not be positive semidefinite. If the $d$ off-diagonal entries are specified independently from any distribution with density bounded by $K$, the probability of validity is at most $K^d\operatorname{Vol}(\mathcal{E}_n)$, which vanishes super-exponentially by Proposition~\ref{prop:VolumeCn}. Entries of fixed size can be valid only if they are coordinated, as in an equicorrelation matrix.

Throughout this subsection we write an entrywise specification as $C_0+E$, where the target $C_0\in\mathcal{E}_n$ is the intended correlation matrix and $E$ is the error. The identity is the most robust target.

\begin{proposition}\label{prop:identitycenter}
Let $E$ be a symmetric hollow random matrix, that is, with zero diagonal, whose off-diagonal entries $(E_{ij})_{i<j}$ have a log-concave joint density that is invariant under $E\mapsto SES$ for every $S=\operatorname{diag}(s_1,\ldots,s_n)$ with $s_i\in\{-1,1\}$. Then $\Pr(C_0+E\in\mathcal{E}_n)\leq\Pr(I_n+E\in\mathcal{E}_n)$ for every $C_0\in\mathcal{E}_n$. In particular, for independent entries uniform on $[-r,r]$, the fraction of the elementwise cube $\{C_0+H:H_{ii}=0,|H_{ij}|\leq r\}$ that consists of correlation matrices is at most $\operatorname{Vol}(\mathcal{E}_n(r))/(2r)^{d}$, the fraction at $C_0=I_n$.
\end{proposition}

For fixed $r$, and more generally for $r\gg\sqrt{\log n/n}$, $\operatorname{Vol}(\mathcal{E}_n(r))/\operatorname{Vol}(\mathcal{E}_n)\rightarrow1$ by Proposition~\ref{prop:entrywise} and Corollary~\ref{cor:gumbel}, so the largest valid fraction of a cube of half-width $r$ around any correlation matrix is $\operatorname{Vol}(\mathcal{E}_n)(2r)^{-d}(1+o(1))=(w/r)^d(1+o(1))$, with $w\approx1.61/\sqrt{n}$ the volume-equivalent half-width of Section~\ref{sec:volume}. For $r$ of order $n^{-1/2}$ the cube captures a vanishing fraction of the elliptope, and $(w/r)^d$ is only an upper bound. It is still super-exponentially small once $r$ exceeds $w$ by a fixed factor. The cube around $I_n$ consists entirely of correlation matrices if and only if $r\leq1/(n-1)$, by Gershgorin's theorem, with the all-$(-r)$ perturbation as the extremal case. It consists mostly of correlation matrices when $r\ll n^{-1/2}$, as shown at the end of this subsection. It contains most of the elliptope once $r\gg\sqrt{\log n/n}$.

The standard remedy is to replace $\tilde{C}$ by the nearest correlation matrix, $C^{*}=\arg\min_{C\in\mathcal{E}_n}\|\tilde{C}-C\|_F$, see \citet{Higham:2002}. The following result shows how much of the perturbation the repair removes asymptotically. It covers an arbitrary target $C_{0,n}$ perturbed by i.i.d.\ entrywise noise, with the entrywise specification above as the case $C_{0,n}=I_n$.

\begin{theorem}\label{thm:nearest}
Let $E_n$ be the symmetric hollow $n\times n$ matrix whose off-diagonal entries $E_{ij}\in[-1,1]$, $1\leq i<j<\infty$, are i.i.d.\ with mean zero and variance $\sigma^2>0$, and let $\tilde{C}_n=C_{0,n}+E_n$ for an arbitrary correlation matrix $C_{0,n}\in\mathcal{E}_n$. Let $C_n^{*}$ be the nearest correlation matrix to $\tilde{C}_n$ in the Frobenius norm. Then
$$
\frac{\|\tilde{C}_n-C_n^{*}\|_F^2}{\|E_n\|_F^2}\overset{\mathrm{as}}{\rightarrow}1
\qquad\text{and}\qquad
\frac{\|C_n^{*}-C_{0,n}\|_F^2}{\|E_n\|_F^2}\overset{\mathrm{as}}{\rightarrow}0.
$$
\end{theorem}

Let $R_n=\|\tilde{C}_n-C_n^{*}\|_F^2/\|E_n\|_F^2$ denote the repair fraction in the first limit. Combined, the two limits say that the projection removes asymptotically all of the squared perturbation and leaves an error that is negligible relative to it. Since $\mathcal{E}_n$ is convex, the projection satisfies $\langle\tilde{C}_n-C_n^{*},C_{0,n}-C_n^{*}\rangle\leq0$, whence $\|C_n^{*}-C_{0,n}\|_F^2\leq(1-R_n)\|E_n\|_F^2$, so the first limit implies the second. The projected matrix is consistent for $C_{0,n}$ in average squared error per entry, with $n^{-2}\|C_n^{*}-C_{0,n}\|_F^2=O_{\mathrm{as}}(\sigma n^{-1/2})$, whereas each unprojected off-diagonal entry has mean squared error $\sigma^2$. The consistency is an average over the $d$ entries, not an entrywise property. The theorem controls the aggregate squared error and says nothing about how the remaining error is distributed among entries. Both limits are scale-free. The noise level $\sigma$ enters only the rates, since the proof gives $1-R_n\leq8(1+o(1))/(\sigma\sqrt{n})$ and $\|C_n^{*}-C_{0,n}\|_F^2/\|E_n\|_F^2\leq4(1+o(1))/(\sigma\sqrt{n})$ almost surely. If the entrywise noise has standard errors of order $T^{-1/2}$, as for pairwise estimates from $T$ observations, the rates are of order $\sqrt{T/n}$. The limits therefore describe the regime $n\gg T$, and for $n\asymp T$ the corrections are of order one. A sample correlation matrix computed jointly from one sample is outside the theorem's setting, since it is already positive semidefinite and its errors are dependent.

The mechanism is a mismatch between norms. The nuclear norm of a correlation matrix equals its trace, so any two correlation matrices are within nuclear-norm distance $2n$ of each other. The perturbation, by contrast, has operator norm of order $\sigma\sqrt{n}$ and squared Frobenius norm of order $\sigma^{2}n^{2}$. By trace duality, its inner product with any feasible displacement is therefore $O(\sigma n^{3/2})$. The proof extends to any perturbation with $n\|E_n\|/\|E_n\|_F^2\rightarrow0$, which requires the perturbation to be both spectrally spread out and large. A low-rank distortion of the same Frobenius norm fails the condition and can pass through the projection. The simulations in Section~\ref{sec:ncmsim} of the Supplement show that the convergence is slow, with corrections of order $(\sigma\sqrt{n})^{-1}$. The bound of the nuclear norm by the trace also underlies \citet[theorem~2.8]{Chatterjee:2015}, who shows that universal singular value thresholding estimates a positive semidefinite matrix with entries in $[-1,1]$, a correlation matrix in particular, from a fully observed version with bounded independent noise, with mean squared error per entry of order $n^{-1/2}$. Theorem~\ref{thm:nearest} concerns the nearest-correlation projection, the standard repair in practice, and gives almost-sure limits for both the repair fraction and the recovery error.

The theorem also identifies $n^{-1/2}$ as the natural transition scale for the specified entries. Consider the scaled specification $\tilde{C}_n=I_n+a_nE_n$, with $E_n$ as in Theorem~\ref{thm:nearest} and $a_n>0$ deterministic. The eigenvalues of $a_nE_n$ spread over $\pm2\sigma a_n\sqrt{n}$. If $a_n\sqrt{n}\rightarrow\infty$, the unit diagonal is asymptotically negligible relative to the perturbation, the proof of Theorem~\ref{thm:nearest} applies verbatim, and the repair fraction again tends to one. If $a_n\sqrt{n}\rightarrow0$, then $\lambda_{\min}(\tilde{C}_n)\geq1-2\sigma a_n\sqrt{n}(1+o(1))>0$ eventually, so $\tilde{C}_n$ is a valid correlation matrix with probability tending to one and the repair fraction vanishes. At the critical scale $a_n=c/\sqrt{n}$ the outcome depends on the constant: for $2\sigma c<1$ the matrix is asymptotically valid (for entries uniform on $[-1,1]$, $\sigma^2=1/3$ and the threshold is $c=\sqrt{3}/2$), whereas for $2\sigma c>1$ a positive fraction of the spectrum is negative and the repair fraction is bounded below by a positive amount that increases with $c$. For a general target, Weyl's inequality gives eventual positive definiteness when $2\sigma c<\liminf_n\lambda_{\min}(C_{0,n})$. When the entries also have a symmetric log-concave density, as uniform entries do, Proposition~\ref{prop:identitycenter} bounds the validity probability at any target by that at the identity. The scale $n^{-1/2}$ is therefore the transition for identity-centered specifications, not a universal threshold. Independent entrywise specifications at scales above $n^{-1/2}$ may therefore undergo large adjustments; parameterizations that enforce validity from the outset, such as the unconstrained transformations discussed in the Introduction \citep{ArchakovHansen:Correlation}, avoid them. A related positive-definiteness issue arises when a correlation matrix is hard-thresholded to obtain a sparse estimate; see \citet{GuillotRajaratnam:2012}.

\section{Conclusion}

The requirement that a correlation matrix must be positive semidefinite imposes severe, highly nonlinear constraints on its individual elements. In this paper, we have quantified exactly how restrictive these constraints become in high dimensions, using the classical volume formula for the elliptope, $\mathcal{E}_n$, together with a refined asymptotic expansion. The volume of $\mathcal{E}_n$ has logarithmic leading term $-\tfrac{1}{4}n^2\log n$, meaning the space of valid correlation matrices becomes vanishingly small compared to its enclosing hypercube. Our contribution lies in what this implies for random correlation matrices. The elliptope concentrates around the identity in the entrywise max norm while expanding globally in the Frobenius norm, and the exact representation as an uncentered sample correlation matrix at $T=n+1$ turns this into the Marchenko--Pastur limit: a uniform draw is asymptotically ill-conditioned despite uniformly small correlations. There is an interesting but imperfect analogy between $\mathcal{E}_n$ and high-dimensional unit balls; we discuss this in Section~\ref{sec:unitball} of the Supplement.

This structure matters for high-dimensional models in statistics and finance. In Bayesian modeling, Theorem~\ref{thm:LKJscaling} shows how the LKJ family behaves as the dimension grows. Every fixed $\eta$ gives entrywise variance of order $1/n$, the limiting spectrum changes only when $\eta$ grows at least as fast as $n$, and the Frobenius distance from the identity stays bounded only when $\eta$ grows at least as fast as $n^2$. No choice of $\eta$ avoids the entrywise shrinkage, so a prior that should not shrink individual correlations requires a prior outside the LKJ family. For estimation, Proposition~\ref{prop:precision} shows that the average variance inflation factor diverges in the critical regime $n/T\rightarrow1$. This is the regime in which minimum-variance portfolio weights are known to be unstable, and in which shrinkage, factor structure, or spectral regularization is used.

For correlation networks, the uniform law provides a null model in which every edge comes from one positive semidefinite matrix, in contrast to the independent-edge Erd\H{o}s--R\'enyi benchmark. Theorem~\ref{thm:PPP} shows that the exceedances of any LKJ-distributed correlation matrix converge to a Poisson point process, which gives the Poisson edge count and the Gumbel limit of the largest correlation. For the uniform law this is a case of known point-process theory for sample correlations \citep{HeinyMikoschYslas:2021}. Because distinct entries are exactly pairwise independent under every LKJ law, the proof also gives an explicit $O(n^{-1})$ total-variation bound for finite-dimensional exceedance counts. At the moderate thresholds used in practice the two null models differ, and the uniform law gives the exact inflation of triangle counts. A fuller treatment of inference for correlation networks is left to future work.

Ultimately, the super-exponential volume collapse of $\mathcal{E}_n$ means that unstructured parameterizations of correlation matrices almost never yield a valid matrix in high dimensions. Theorem~\ref{thm:nearest} quantifies what repair then achieves. For bounded, centered i.i.d.\ off-diagonal perturbations, the projection to the nearest valid correlation matrix removes a fraction of the squared perturbation that tends to one, and it recovers the underlying matrix in average squared error per entry. Corollary~\ref{cor:floorprojection} shows that spectral-floor projection requires no new projection algorithm: an affine rescaling reduces it to the ordinary nearest-correlation problem. Independent entrywise specifications may therefore undergo large adjustments, which valid parameterizations avoid; the Generalized Fisher Transformation \citep{ArchakovHansen:Correlation} maps the interior of $\mathcal{E}_n$ bijectively onto an unconstrained space, so optimization and prior specification can proceed without leaving the feasible set.

\appendix
\setcounter{equation}{0}\renewcommand{\theequation}{A.\arabic{equation}}
\setcounter{table}{0}\renewcommand{\thetable}{A.\arabic{table}}

\section{Proofs of Main Results}

This appendix contains the proofs of the main results. The more routine proofs (Propositions~\ref{prop:VolumeCn}, \ref{prop:entrywise}, \ref{prop:C-I}, \ref{prop:WishartGram}, \ref{prop:precision}, and \ref{prop:identitycenter}, and Theorem~\ref{thm:MP}) are given in Sections~\ref{sec:barnes} and~\ref{sec:auxproofs} of the Supplement.

\noindent{}{\bf Proof of Corollary~\ref{cor:floorprojection}.}
By the homothety of Theorem~\ref{thm:specfloor}, $C\in\mathcal{E}_n$ satisfies $\lambda_{\min}(C)\geq x$ if and only if $C=xI_n+(1-x)R$ for a unique $R\in\mathcal{E}_n$, in which case $\|A-C\|_F=(1-x)\bigl\|(A-xI_n)/(1-x)-R\bigr\|_F$. Minimizing over $R\in\mathcal{E}_n$ proves both claims.\hfill$\square$

\noindent{}{\bf Proof of Lemma~\ref{lem:pairwise}.}
By Proposition~\ref{prop:WishartGram}(i), $C\sim\operatorname{LKJ}(\eta)$ equals $D^{-1}WD^{-1}$ with $W\sim W_n(\nu,I_n)$, $\nu=n+2\eta-1$, $D=\operatorname{diag}(W_{11}^{1/2},\ldots,W_{nn}^{1/2})$. The Bartlett decomposition $W=LL^\prime$ \citep[theorem~3.2.14]{Muirhead:1982} has independent rows $l_1,\ldots,l_n$, where $l_i=(L_{i1},\ldots,L_{ii},0,\ldots,0)$ with $L_{ij}\sim N(0,1)$ for $i>j$ and $L_{ii}^2\sim\chi^2_{\nu-i+1}$. Muirhead's proof is a change of variables in the density, so the structure holds for all real $\nu>n-1$; in particular $l_1=(L_{11},0,\ldots,0)$ always, and $C_{ij}=l_i^\prime l_j/(\|l_i\|\|l_j\|)$.

The reduction to canonical index pairs is by distributional invariance, not by permuting a fixed Bartlett factor. Let $P$ be a permutation matrix. Since the $\operatorname{LKJ}(\eta)$ density $\propto\det(C)^{\eta-1}$ and the reference measure on $\mathcal{E}_n$ are both invariant under $C\mapsto PCP^\prime$, we have $PCP^\prime\overset{d}{=}C$; equivalently, $PCP^\prime$ is again the normalized Wishart of Proposition~\ref{prop:WishartGram}(i), because $PWP^\prime\sim W_n(\nu,I_n)$. Applying the Bartlett decomposition afresh to $PWP^\prime$ (rather than permuting the rows of $L$, which would destroy triangularity) yields a new triangular factor with independent rows, in terms of which the pair $(C_{ij},C_{kl})$ has been relabeled to a canonical pair. It therefore suffices to verify independence for the two canonical configurations below.

\textit{Disjoint indices} ($\{i,j\}\cap\{k,l\}=\emptyset$): $C_{ij}$ and $C_{kl}$ are measurable with respect to the disjoint collections $\{l_i,l_j\}$ and $\{l_k,l_l\}$; independence follows from the independence of the rows.

\textit{Shared index} (by relabeling, take $(\{1,2\},\{1,3\})$): since $l_1=(L_{11},0,\ldots,0)$, we have $l_1^\prime l_j=L_{11}L_{j1}$ and $\|l_1\|=L_{11}$, so $C_{1j}=L_{j1}/\|l_j\|$ for $j\geq2$. Hence $C_{12}=L_{21}/\|l_2\|$ is a function of $l_2$ alone, and $C_{13}=L_{31}/\|l_3\|$ is a function of $l_3$ alone; independence follows from the independence of $l_2$ and $l_3$. The shared-index case is the first-row independence stated by \citet{JoeKurowicka:2026}.\hfill$\square$

\noindent{}{\bf Proof of Lemma~\ref{lem:betatail}.}
\textit{Step~1 (exact tail probability).} Since $(C_{12}+1)/2\sim\operatorname{Beta}(\tfrac{\nu_n-1}{2},\tfrac{\nu_n-1}{2})$ under $\operatorname{LKJ}_n(\eta_n)$, with $\nu_n=n+1$ for the uniform law, the marginal of $C_{12}$ has density
\begin{equation}\label{eq:margdens}
f_n(x)=Z_n(1-x^2)^{(\nu_n-3)/2}\quad\text{on }(-1,1),\qquad
Z_n=\frac{\Gamma(\tfrac{\nu_n}{2})}{\sqrt{\pi}\Gamma(\tfrac{\nu_n-1}{2})}\sim\sqrt{\tfrac{\nu_n}{2\pi}}.
\end{equation} Set $\tau=\tau_n(t)=\sqrt{(\mu_n+t)/\nu_n}$, so that $\nu_nC_{12}^2-\mu_n>t$ iff $|C_{12}|>\tau$. Since $\nu_n>n-1$, $\tau^2=(\mu_n+t)/\nu_n=O(\log n/n)\to0$ and $\nu_n\tau^4=(\mu_n+t)^2/\nu_n\rightarrow0$, while
\begin{equation}\label{eq:endpointrate}
\nu_n\tau_n^2=\mu_n+t\sim4\log n\rightarrow\infty.
\end{equation}
By symmetry,
$$
\Pr(|C_{12}|>\tau) = 2Z_n\int_\tau^1(1-x^2)^{(\nu_n-3)/2}dx.
$$
\textit{Step~2 (endpoint approximation, uniformly in the triangular array).} Both the endpoint $\tau_n$ and the local decay rate depend on $n$, so we give the endpoint expansion in the scaled variable
$$
\delta=\frac{1-\tau^2}{(\nu_n-3)\tau}u,\qquad u\in[0,U_n],\qquad U_n=\frac{(\nu_n-3)\tau}{1+\tau}\rightarrow\infty,
$$
where $x=\tau+\delta$ ranges over $[\tau,1]$ as $u$ ranges over $[0,U_n]$. Writing $1-x^2=(1-\tau^2)\bigl(1-(2\tau \delta+\delta^2)/(1-\tau^2)\bigr)$, the normalized integrand is
$$
g_n(u):=\frac{(1-x^2)^{(\nu_n-3)/2}}{(1-\tau^2)^{(\nu_n-3)/2}}
=\exp\left\{\tfrac{\nu_n-3}{2}\log\left(1-\frac{2u}{\nu_n-3}-\frac{\delta^2}{1-\tau^2}\right)\right\},\qquad 0\le u\le U_n,
$$
extended by $g_n(u)=0$ for $u>U_n$. Two facts complete the argument.

\emph{(a) Pointwise limit on bounded $u$-sets.} Since $\delta^2/(1-\tau^2)\leq u^2/[(\nu_n-3)^2\tau^2]$ and, by (\ref{eq:endpointrate}), $(\nu_n-3)\tau^2\rightarrow\infty$, we have $\tfrac{\nu_n-3}{2}\cdot \delta^2/(1-\tau^2)=O(u^2/(\nu_n\tau^2))=o(1)$ for $u$ in bounded sets; the quadratic remainder of the logarithm contributes $O(u^2/\nu_n)=o(1)$ likewise. Hence $\tfrac{\nu_n-3}{2}\log(\cdot)=-u+o(1)$ and $g_n(u)\rightarrow e^{-u}$ for every $u\geq0$. This is precisely where the condition $\nu_n\tau_n^2\rightarrow\infty$ enters: it makes the quadratic term negligible on the endpoint scale $\delta=O((\nu_n\tau_n)^{-1})$.

\emph{(b) Integrable domination.} Because $x\mapsto\log(1-x^2)$ is concave, its tangent at $\tau$ is a global upper bound, $1-x^2\leq(1-\tau^2)\exp\{-2\tau(x-\tau)/(1-\tau^2)\}$ for $x\geq\tau$, whence
$$
g_n(u)\leq\exp\left\{-\frac{(\nu_n-3)\tau \delta}{1-\tau^2}\right\}=e^{-u}\qquad\text{for all }u\geq0\text{ and all }n\geq4.
$$
By dominated convergence, $\int_0^\infty g_n(u)du\rightarrow\int_0^\infty e^{-u}du=1$, and changing variables back,
$$
\int_\tau^1(1-x^2)^{(\nu_n-3)/2}dx=\frac{(1-\tau^2)^{(\nu_n-1)/2}}{(\nu_n-3)\tau}\int_0^\infty g_n(u)du\sim\frac{(1-\tau^2)^{(\nu_n-1)/2}}{(\nu_n-3)\tau}.
$$
\textit{Step~3 (substitution and limit).} Since $Z_n\sim\sqrt{\nu_n/(2\pi)}$ by Stirling's formula and $\nu_n-3\sim\nu_n$,
$$
\Pr(|C_{12}|>\tau) \sim \frac{2\sqrt{\nu_n/(2\pi)}(1-\tau^2)^{(\nu_n-1)/2}}{\nu_n\tau} = \frac{2(1-\tau^2)^{(\nu_n-1)/2}}{\sqrt{2\pi\nu_n}\tau}.
$$
Because $\tau^2\to0$ and $\nu_n\tau^4\to0$, we have
$$
(1-\tau^2)^{(\nu_n-1)/2}=\exp\{\tfrac{\nu_n-1}{2}\log(1-\tau^2)\}=\exp\{-\tfrac{\nu_n\tau^2}{2}(1+O(\tau^2)+O(\nu_n^{-1}))\}\sim e^{-(\mu_n+t)/2},
$$
while $\sqrt{\nu_n}\tau=\sqrt{\mu_n+t}$. Therefore $\Pr(|C_{12}|>\tau)\sim 2e^{-(\mu_n+t)/2}/\sqrt{2\pi(\mu_n+t)}$, and
$$
d\cdot\Pr(|C_{12}|>\tau) \sim \frac{n(n-1)/2\cdot 2e^{-(\mu_n+t)/2}}{\sqrt{2\pi(\mu_n+t)}} \sim \frac{n^2e^{-(\mu_n+t)/2}}{\sqrt{2\pi(\mu_n+t)}}.
$$
Substituting $\mu_n=4\log n-\log\log n$: $e^{-\mu_n/2}=e^{-2\log n+(\log\log n)/2}=n^{-2}\sqrt{\log n}$ and $\sqrt{2\pi(\mu_n+t)}\sim\sqrt{8\pi\log n}$, so
$$
d\cdot\Pr\{\nu_nC_{12}^2-\mu_n>t\} \sim \frac{n^2\cdot n^{-2}\sqrt{\log n}\cdot e^{-t/2}}{\sqrt{8\pi\log n}} = \frac{e^{-t/2}}{\sqrt{8\pi}}.
$$
\textit{Step~4 (density and Borel sets).} The random variable $\nu_nC_{12}^2-\mu_n$ is supported on $[-\mu_n,\nu_n-\mu_n]$, and the pre-limit mean measure $\lambda_n$ has Lebesgue density
$$
h_n(t)=\frac{d\cdot Z_n\bigl(1-\tau_n(t)^2\bigr)^{(\nu_n-3)/2}}{\nu_n\tau_n(t)},\qquad \tau_n(t)=\sqrt{\tfrac{\mu_n+t}{\nu_n}},
$$
on that interval, obtained from (\ref{eq:margdens}) by the change of variables $t\mapsto\tau_n(t)$ and the symmetry of $f_n$. Applying the estimates of Step~3 with $Z_n/(\nu_n\tau_n(t))\sim1/\sqrt{2\pi(\mu_n+t)}$ gives $h_n(t)\rightarrow\tfrac{1}{2\sqrt{8\pi}}e^{-t/2}$, the density of $\Lambda$, pointwise.

For the domination, fix $t_0\in\mathbb{R}$ and let $t\in[t_0,\nu_n-\mu_n]$, so that $y:=\tau_n(t)^2=(\mu_n+t)/\nu_n\in(0,1)$. From $\log(1-y)\leq-y$,
$$
(1-y)^{(\nu_n-3)/2}\leq\exp\{-\tfrac{\nu_n-3}{2}y\}
=\exp\left\{-\frac{\mu_n+t}{2}\cdot\frac{\nu_n-3}{\nu_n}\right\}
=e^{-(\mu_n+t)/2}\exp\left\{\frac{3(\mu_n+t)}{2\nu_n}\right\}.
$$
On the support, $\mu_n+t\leq\nu_n$, so the last factor is at most $e^{3/2}$, uniformly in $n$ and in $t$. Because $\tau_n(t)^2$ can approach one near the upper endpoint of the support, no uniform positive lower bound on $1-\tau_n(t)^2$ is available; the preceding exponential bound remains valid throughout the support. Since $e^{-\mu_n/2}=n^{-2}\sqrt{\log n}$ and $Z_n/(\nu_n\tau_n(t))\leq Z_n/(\nu_n\tau_n(t_0))\sim1/\sqrt{2\pi\mu_n}$ for $t\geq t_0$, we obtain
$$
h_n(t)\leq e^{3/2}\frac{dZ_ne^{-\mu_n/2}}{\nu_n\tau_n(t_0)}e^{-t/2}\leq K e^{-t/2},\qquad t\in[t_0,\nu_n-\mu_n],
$$
for a constant $K=K(t_0)$ that does not depend on $n$, because $d\cdot n^{-2}\sqrt{\log n}/\sqrt{2\pi\mu_n}=O(1)$ by $\mu_n\sim4\log n$.

For any Borel set $B$ with $t_0:=\inf B>-\infty$ and $\Lambda(\partial B)=0$, the domination $h_n\leq Ke^{-t/2}$ holds on all of $B$, and dominated convergence gives $\lambda_n(B)=\int_B h_n(t)dt\rightarrow\Lambda(B)$. The restriction to sets bounded away from $-\infty$ is essential: the domination fails near the lower endpoint $-\mu_n$ of the support, where the pre-limit mean measure is unbounded.\hfill$\square$

\noindent{}{\bf Proof of Theorem~\ref{thm:PPP}.}
We apply the Chen--Stein method of \citet{ArratiaGoldsteinGordon:1989} with an explicit dependency graph. The key structural input is Lemma~\ref{lem:pairwise}: by exact pairwise independence, every joint-expectation term $\mathbb{E}[I_\alpha I_\beta]$ with $\alpha\neq\beta$ reduces to the product $p_\alpha p_\beta$; this does not make $b_2$ zero, but it makes $b_2\leq b_1=O(n^{-1})$.

For disjoint Borel sets $B_1,\ldots,B_m$ bounded away from $-\infty$ with $\Lambda(\partial B_r)=0$, where $m\geq1$, index the indicators by pairs $(\alpha,r)$, with $\alpha=(i,j)$ an edge and $r\in\{1,\ldots,m\}$ identifying the set $B_r$:
$$
I_{(\alpha,r)}=\mathbf{1}[\nu_nC_{ij}^2-\mu_n\in B_r],\qquad p_{(\alpha,r)}=\mathbb{E}[I_{(\alpha,r)}]=\lambda_n(B_r)/d,
$$
so that $\Xi_n(B_r)=\sum_\alpha I_{(\alpha,r)}$ with $\mathbb{E}[\Xi_n(B_r)]=\lambda_n(B_r)\rightarrow\Lambda(B_r)$ by Lemma~\ref{lem:betatail}. Following the convention of \citet{ArratiaGoldsteinGordon:1989}, take the dependency neighborhood
$$
B_{(\alpha,r)}=\{(\beta,s):\beta\text{ shares a vertex with }\alpha,\ 1\leq s\leq m\},\qquad |B_{(\alpha,r)}|=m(2n-3),
$$
which contains $(\alpha,r)$ itself. In the Bartlett representation of the proof of Lemma~\ref{lem:pairwise}, each entry $C_{kl}=l_k^\prime l_l/(\|l_k\|\|l_l\|)$ is a function of the rows $l_k$ and $l_l$, and the rows are independent, so $\{I_{(\beta,s)}:(\beta,s)\notin B_{(\alpha,r)}\}$ is measurable with respect to $\sigma(l_q:q\notin\{i,j\})$, which is independent of $\sigma(l_i,l_j)$ and hence of $I_{(\alpha,r)}$; therefore $b_3=0$. This step uses the independence of the rows, not merely pairwise independence of the entries. The joint expectations are elementary: if $\alpha=\beta$ and $r\neq s$, then $I_{(\alpha,r)}I_{(\alpha,s)}=0$ because $B_r\cap B_s=\emptyset$; if $\alpha\neq\beta$, Lemma~\ref{lem:pairwise} gives $\mathbb{E}[I_{(\alpha,r)}I_{(\beta,s)}]=\lambda_n(B_r)\lambda_n(B_s)/d^2$. Since $\lambda_n(B_r)\rightarrow\Lambda(B_r)$ implies $p_n^{\max}=\max_r\lambda_n(B_r)/d=O(n^{-2})$,
\begin{align*}
b_1&=\sum_{(\alpha,r)}\sum_{(\beta,s)\in B_{(\alpha,r)}}p_{(\alpha,r)}p_{(\beta,s)}\leq dm^2(2n-3)(p_n^{\max})^2=O(n^{-1}),\\
b_2&=\sum_{(\alpha,r)}\sum_{(\beta,s)\in B_{(\alpha,r)},(\beta,s)\neq(\alpha,r)}\mathbb{E}[I_{(\alpha,r)}I_{(\beta,s)}]\leq b_1,
\end{align*}
so Theorem~2 of \citet{ArratiaGoldsteinGordon:1989}, the process version of the Chen--Stein bound, gives, with $d_{\mathrm{TV}}(P,Q)=\sup_A|P(A)-Q(A)|$,
$$
d_{\mathrm{TV}}\Bigl(\bigl(\Xi_n(B_1),\ldots,\Xi_n(B_m)\bigr),{\textstyle\bigotimes_r}\operatorname{Poisson}(\lambda_n(B_r))\Bigr)\leq2(b_1+b_2)=O(n^{-1})\rightarrow0;
$$
for $m=1$, Theorem~1 of \citet{ArratiaGoldsteinGordon:1989} gives the sharper scalar bound $d_{\mathrm{TV}}(\Xi_n(B),\operatorname{Poisson}(\lambda_n(B)))\leq b_1+b_2+b_3=O(n^{-1})$. Using $\lambda_n(B_r)\rightarrow\Lambda(B_r)$ and the continuity of the Poisson family in its mean, $d_{\mathrm{TV}}(\operatorname{Poisson}(\lambda),\operatorname{Poisson}(\lambda^\prime))\leq|\lambda-\lambda^\prime|$, the triangle inequality gives
$$
d_{\mathrm{TV}}\Bigl(\bigl(\Xi_n(B_1),\ldots,\Xi_n(B_m)\bigr),{\textstyle\bigotimes_r}\operatorname{Poisson}(\Lambda(B_r))\Bigr)\rightarrow0,
$$
so $(\Xi_n(B_1),\ldots,\Xi_n(B_m))$ converges jointly in distribution to independent $\operatorname{Poisson}(\Lambda(B_r))$ variables. Since these finite-dimensional distributions converge, and Lemma~\ref{lem:betatail} gives $\mathbb{E}[\Xi_n(B)]=\lambda_n(B)\to\Lambda(B)$, for every finite union $B$ of intervals bounded away from $-\infty$ with $\Lambda(\partial B)=0$, Kallenberg's theorem \citep[proposition~3.22]{Resnick:1987} gives $\Xi_n\Rightarrow\operatorname{PPP}(\Lambda)$, where $\Lambda$ is extended to $(-\infty,\infty]$ by $\Lambda(\{\infty\})=0$.\hfill$\square$

\noindent{}{\bf Proof of Corollary~\ref{cor:ppp_consequences}.}
(a): Let $t_n=\nu_n\tau_n^2-\mu_n$, so that $N_n(\tau_n)=\Xi_n(t_n,\infty)$ and $\mathbb{E}N_n(\tau_n)=\lambda_n(t_n,\infty)$, and let $t_\lambda=-2\log(\lambda\sqrt{8\pi})$ be the unique solution of $\Lambda(t_\lambda,\infty)=\lambda$. We first show $t_n\to t_\lambda$. The map $t\mapsto\lambda_n(t,\infty)$ is nonincreasing and converges pointwise to the continuous, strictly decreasing limit $t\mapsto\Lambda(t,\infty)$ (Lemma~\ref{lem:betatail}); if $t_n\not\to t_\lambda$, then along a subsequence $t_n\geq t_\lambda+\varepsilon$ (or $\leq t_\lambda-\varepsilon$), and monotonicity gives $\limsup\lambda_n(t_n,\infty)\leq\Lambda(t_\lambda+\varepsilon,\infty)<\lambda$ (respectively $\liminf\geq\Lambda(t_\lambda-\varepsilon,\infty)>\lambda$), contradicting $\mathbb{E}N_n(\tau_n)\to\lambda$. Given $t_n\to t_\lambda$, for any $\varepsilon>0$ eventually $\Xi_n(t_\lambda+\varepsilon,\infty)\leq N_n(\tau_n)\leq\Xi_n(t_\lambda-\varepsilon,\infty)$, and Theorem~\ref{thm:PPP} gives $\Xi_n(t_\lambda\pm\varepsilon,\infty)\overset{d}{\rightarrow}\operatorname{Poisson}(\Lambda(t_\lambda\pm\varepsilon,\infty))$. Letting $\varepsilon\downarrow0$ and using the continuity of $t\mapsto\Lambda(t,\infty)$ yields $N_n(\tau_n)\overset{d}{\rightarrow}\operatorname{Poisson}(\lambda)$. The sufficient condition in the statement gives $t_n=t_\lambda+o(1)$ directly; note that a relative error of order $o(1)$ in $\tau_n^2$ would shift $t_n$ by $o(\mu_n)=o(\log n)$, which is why the additive form is required.
(b): By Lemma~\ref{lem:pairwise}, two edges of a triangle are present with probability $p_n(\tau_n)^2$, so $\mathbb{E}[\#\text{triangles}]\leq\binom{n}{3}p_n(\tau_n)^2=O(n^3(\lambda/d)^2)=O(n^{-1})\rightarrow0$; apply Markov.
\hfill$\square$

\noindent{}{\bf Proof of Corollary~\ref{cor:gumbel}.}
$\nu_nM_n^2-\mu_n$ is the largest atom of $\Xi_n$, so for fixed $t\in\mathbb{R}$,
$$\Pr\{\nu_nM_n^2-\mu_n\leq t\}=\Pr\{\Xi_n(t,\infty)=0\}\rightarrow\exp\{-\Lambda(t,\infty)\}=\mathcal{G}(t),$$
by Theorem~\ref{thm:PPP}. The in-probability claim follows from $\mu_n=4\log n-\log\log n\sim4\log n$.\hfill$\square$

\noindent{}{\bf Proof of Theorem~\ref{thm:LKJscaling}.}
By Proposition~\ref{prop:WishartGram}(i), we may write $C_n=D^{-1}WD^{-1}$ with $W\sim W_n(\nu_n,I_n)$ and $D=\operatorname{diag}(\sigma_1,\ldots,\sigma_n)$, $\sigma_i^2=W_{ii}$. The Bartlett decomposition $W=LL^\prime$, with $L$ lower triangular, $L_{ij}\sim N(0,1)$ for $i>j$, and $L_{ii}^2\sim\chi^2_{\nu_n-i+1}$, holds for all real $\nu_n>n-1$ (see the proof of Lemma~\ref{lem:pairwise}).

(i) We prove the Marchenko--Pastur limit in two steps, for every real $\nu_n>n-1$; we use the in-probability form throughout, since the law of $C_n$ changes with $n$.

\textit{Step~1 (coupling to an integer-degree Wishart).} For real $\nu_n$, let $m_n=\lceil\nu_n\rceil$. Augment the Bartlett factor: if $m_n>\nu_n$ let $G_i\overset{\mathrm{ind}}{\sim}\operatorname{Gamma}(\tfrac{m_n-\nu_n}{2},2)$ be independent of $L$, and if $m_n=\nu_n$ set $G_i=0$; set $\tilde{L}_{ii}=\sqrt{L_{ii}^2+G_i}$ and $\tilde{L}_{ij}=L_{ij}$ for $i>j$, and define $\tilde{W}=\tilde{L}\tilde{L}^\prime\sim W_n(m_n,I_n)$. Writing $\Delta=\operatorname{diag}(\tilde{L}_{11}-L_{11},\ldots,\tilde{L}_{nn}-L_{nn})\geq0$ so that $\tilde{L}=L+\Delta$,
$$
\tilde{W}-W = \Delta L^\prime + L\Delta + \Delta^2.
$$
Since $(\sqrt{a+b}-\sqrt{a})^2\leq b$ for $a,b\geq0$, we have $\|\Delta\|_F^2\leq\sum_iG_i=O_p(n)$, because $\mathbb{E}\sum_iG_i=n(m_n-\nu_n)\leq n$. To bound $\|L\|$ without circularity, start from the integer-degree factor: $\|\tilde{L}\|^2=\|\tilde{L}\tilde{L}^\prime\|=\|\tilde{W}\|=O_p(m_n)$ by the standard largest-eigenvalue bound for integer Wishart matrices \citep[theorem~5.8]{BaiSilverstein:2010}, and $\|\Delta\|\leq\|\Delta\|_F=O_p(\sqrt{n})$, so the triangle inequality gives $\|L\|\leq\|\tilde{L}\|+\|\Delta\|=O_p(\sqrt{\nu_n})$ (using $m_n/\nu_n\to1$ and $n\leq\nu_n+1$). Then $\|\tilde{W}-W\|_F\leq2\|\Delta\|_F\|L\|+\|\Delta\|_F^2=O_p(\sqrt{n\nu_n})$, hence
$$
\frac{1}{n}\Bigl\|\frac{\tilde{W}-W}{\nu_n}\Bigr\|_F^2 = O_p\Bigl(\frac{1}{\nu_n}\Bigr)\to0.
$$
By the Hoffman--Wielandt inequality \citep[corollary~A.41]{BaiSilverstein:2010}, the cube of the L\'evy distance between the empirical spectral distributions of $W/\nu_n$ and $\tilde{W}/\nu_n$ is bounded by this quantity and tends to zero in probability. Since $n/m_n\to c_\theta$, the standard Marchenko--Pastur theorem for sample covariance matrices \citep[theorem~3.6]{BaiSilverstein:2010} gives that the empirical spectral distribution of $\tilde{W}/m_n$ converges weakly, in probability, to $F_{\mathrm{MP},c_\theta}$. Because $m_n/\nu_n\to1$, the same holds for $\tilde{W}/\nu_n$, and combined with the Hoffman--Wielandt bound above, the empirical spectral distribution of $W/\nu_n$ also converges to $F_{\mathrm{MP},c_\theta}$ in probability.

\textit{Step~2 (diagonal normalization).} It remains to pass from $\bar{W}=W/\nu_n$ to $C_n=D^{-1}WD^{-1}$, where $D=\operatorname{diag}(\sigma_1,\ldots,\sigma_n)$ with $\sigma_i^2=W_{ii}\sim\chi^2_{\nu_n}$. Since $W_{11},\ldots,W_{nn}$ are independent $\chi^2_{\nu_n}$, the standard chi-squared tail bound gives $\Pr(|W_{ii}/\nu_n-1|>\varepsilon)\leq 2e^{-c_0\nu_n\varepsilon^2}$ for a universal constant $c_0>0$. A union bound over $i=1,\ldots,n$ then gives $\Pr(\max_i|D_{ii}^2/\nu_n-1|>\varepsilon)\leq 2ne^{-c_0\nu_n\varepsilon^2}$, which tends to zero whenever $\nu_n\gg\log n$; this holds throughout our parameter range since $\nu_n=n+2\eta_n-1\geq n-1$. Hence $\max_i|D_{ii}^2/\nu_n-1|\overset{p}{\to}0$ and $\max_i|D_{ii}/\sqrt{\nu_n}-1|\overset{p}{\to}0$. Write $D=\sqrt{\nu_n}\,Q$ with $Q=\operatorname{diag}(D_{11}/\sqrt{\nu_n},\ldots,D_{nn}/\sqrt{\nu_n})$, so $C_n=Q^{-1}\bar{W}Q^{-1}$. For any $\varepsilon>0$, with probability tending to one every diagonal entry of $Q$ lies in $[1-\varepsilon,1+\varepsilon]$. We compare eigenvalues of $C_n$ and $\bar{W}$ via Courant--Fischer. Setting $y=Q^{-1}x$ gives $x^\prime C_n x=y^\prime\bar{W}y$, and as $x$ ranges over a $k$-dimensional subspace $\mathcal{S}$, $y$ ranges over the $k$-dimensional subspace $Q^{-1}\mathcal{S}$. Since the diagonal entries of $Q$ lie in $[1-\varepsilon,1+\varepsilon]$,
$$
(1-\varepsilon)^2\|y\|^2\leq\|Qy\|^2=\|x\|^2\leq(1+\varepsilon)^2\|y\|^2.
$$
By the Courant--Fischer max-min representation, for each $k$,
$$
\lambda_k(C_n)
= \max_{\dim\mathcal{S}=k}\min_{x\in\mathcal{S}\setminus\{0\}}\frac{y^\prime\bar{W}y}{\|x\|^2}
\in\left[\frac{\lambda_k(\bar{W})}{(1+\varepsilon)^2},\frac{\lambda_k(\bar{W})}{(1-\varepsilon)^2}\right].
$$
Hence every ordered eigenvalue of $C_n$ is within a factor $(1\pm\varepsilon)^{\pm2}$ of the corresponding eigenvalue of $\bar{W}$, so the empirical spectral distributions of $C_n$ and $\bar{W}$ have the same weak limit. (This diagonal-normalization step is the content of the comparison theorem of \citet[theorem~2.1]{Heiny:2022} in a general covariance setting; our argument specializes to $\Sigma=I_n$.) Combining Steps~1 and~2 yields part~(i).

(ii) Write $\|C_n-I_n\|_F^2=2\sum_{i<j}C_{ij}^2$. The marginal distribution of $C_{ij}$ satisfies $(C_{ij}+1)/2\sim \operatorname{Beta}(\tfrac{\nu_n-1}{2},\tfrac{\nu_n-1}{2})$, so that $\mathbb{E}[C_{ij}^2]=1/\nu_n$ and $\mathbb{E}[C_{ij}^4]=3/(\nu_n(\nu_n+2))$, and the exact mean $\mathbb{E}\|C_n-I_n\|_F^2=2d/\nu_n=n(n-1)/\nu_n$ follows. For the concentration claim, Lemma~\ref{lem:pairwise} gives pairwise independence of distinct entries, so $\operatorname{cov}(C_{ij}^2,C_{kl}^2)=0$ for $\{i,j\}\neq\{k,l\}$ and $\operatorname{var}(\sum_{i<j}C_{ij}^2)=d\operatorname{var}(C_{12}^2)=d\bigl(\tfrac{3}{\nu_n(\nu_n+2)}-\tfrac{1}{\nu_n^2}\bigr)$; the variance-to-squared-mean ratio is therefore exactly $2(\nu_n-1)/(d(\nu_n+2))<2/d\rightarrow0$, and Chebyshev's inequality gives $\|C_n-I_n\|_F^2/\mathbb{E}\|C_n-I_n\|_F^2\overset{p}{\rightarrow}1$. The stated consequences follow from $\mathbb{E}\|C_n-I_n\|_F^2=n(n-1)/\nu_n$.\hfill$\square$

\noindent{}{\bf Proof of Theorem~\ref{thm:nearest}.}
For $C\in\mathcal{E}_n$, write $\tilde{C}_n-C=E_n-A$ with $A=C-C_{0,n}$. The nuclear norm of a positive semidefinite matrix equals its trace, so $\|A\|_{*}\leq\|C\|_{*}+\|C_{0,n}\|_{*}=2n$, and trace duality gives $|\langle E_n,A\rangle|\leq\|E_n\|\|A\|_{*}\leq2n\|E_n\|$. Hence, uniformly over $C\in\mathcal{E}_n$,
$$
\|\tilde{C}_n-C\|_F^2=\|E_n\|_F^2-2\langle E_n,A\rangle+\|A\|_F^2\geq\|E_n\|_F^2-4n\|E_n\|,
$$
while taking $C=C_{0,n}$ shows $\|\tilde{C}_n-C_n^{*}\|_F^2\leq\|E_n\|_F^2$. In terms of the repair fraction $R_n$ of Section~\ref{sec:nearest}, these are the deterministic bounds $1-4n\|E_n\|/\|E_n\|_F^2\leq R_n\leq1$. The matrix $E_n$ is a Wigner matrix with i.i.d.\ bounded entries, so $\|E_n\|=(2\sigma+o(1))\sqrt{n}$ almost surely, by \citet{BaiYin:1988} applied to $\pm E_n$, while $\|E_n\|_F^2=\sum_{i\neq j}E_{ij}^2=\sigma^2n^2(1+o_{\mathrm{as}}(1))$ by the law of large numbers; hence $0\leq1-R_n\leq\tfrac{8}{\sigma\sqrt{n}}(1+o_{\mathrm{as}}(1))$, which proves the first limit. For the second, write $A_n=C_n^{*}-C_{0,n}$; the projection inequality $\langle\tilde{C}_n-C_n^{*},C_{0,n}-C_n^{*}\rangle\leq0$ reads $\langle E_n-A_n,A_n\rangle\geq0$, so $\|A_n\|_F^2\leq\langle E_n,A_n\rangle\leq2n\|E_n\|$, whence $n^{-2}\|A_n\|_F^2\leq(4\sigma+o(1))n^{-1/2}$ almost surely and the second limit follows.\hfill$\square$

\section*{Acknowledgments}
I thank seminar participants at the University of Illinois Urbana--Champaign and participants at the Vienna--Copenhagen Conference on Financial Econometrics (VieCo 2026), University of Vienna, for helpful comments.

\section*{Use of AI tools}
Generative AI tools were used for manuscript review and copy-editing, and to assist with numerical checks of the derivations. All mathematical statements, proofs, and conclusions are the author's own and have been verified by the author.

\bibliographystyle{apalike}
\bibliography{prh}

\clearpage
\setcounter{equation}{0}\renewcommand{\theequation}{S.\arabic{equation}}
\setcounter{figure}{0}\renewcommand{\thefigure}{S.\arabic{figure}}
\setcounter{table}{0}\renewcommand{\thetable}{S.\arabic{table}}
\setcounter{section}{0}\renewcommand{\thesection}{S.\arabic{section}}
\setcounter{page}{1}\renewcommand{\thepage}{S.\arabic{page}}
\setcounter{lemma}{0}\renewcommand{\thelemma}{S.\arabic{lemma}}
\setcounter{proposition}{0}\renewcommand{\theproposition}{S.\arabic{proposition}}

\section*{Supplementary Material}\label{sec:supp}
\setcounter{section}{0}

This supplement contains a comparison of the volume of $\mathcal{E}_n$ with that of the unit ball (Section~\ref{sec:unitball}), the proof of Proposition~\ref{prop:VolumeCn} including the Barnes $G$-function expansion (Section~\ref{sec:barnes}), the proofs of Propositions~\ref{prop:entrywise}, \ref{prop:C-I}, \ref{prop:WishartGram}, \ref{prop:precision}, and \ref{prop:identitycenter} and Theorem~\ref{thm:MP} (Section~\ref{sec:auxproofs}), exact closed-form expressions for $\operatorname{Vol}(\mathcal{E}_n)$ in selected dimensions up to $n=30$ together with two-step recursions and an exact log-volume recursion (Section~\ref{sec:exact}), an illustration of the point-process limit of Theorem~\ref{thm:PPP} (Section~\ref{sec:pppfig}), an exact computation of triangle counts at moderate thresholds (Section~\ref{sec:triangles}), simulation evidence on the nearest-correlation repair fraction of Theorem~\ref{thm:nearest} (Section~\ref{sec:ncmsim}), and a note on the centering in the log-determinant CLT of \citet{HaneaNane:2018} (Section~\ref{sec:centering}).

\section{Comparison with Volume of the Unit Ball}\label{sec:unitball}
The unit ball $\mathbb{B}_k$ in $\mathbb{R}^k$ has volume $\operatorname{Vol}(\mathbb{B}_k)=\pi^{k/2}/\Gamma(\tfrac{k}{2}+1)$. It follows that the volumes of the elliptope and the unit ball coincide for $\operatorname{Vol}(\mathcal{E}_2)=\operatorname{Vol}(\mathbb{B}_1)=2$ and $\operatorname{Vol}(\mathcal{E}_3)=\operatorname{Vol}(\mathbb{B}_4)=\pi^2/2$. We know of no other integer-dimensional coincidences.

There are two ways to compare the volume of elliptopes to that of unit balls. If we compare $\mathcal{E}_n$ to $\mathbb{B}_n$, whose volumes live in different dimensions, we find that $\mathcal{E}_n$ decays much faster. While $\log \operatorname{Vol}(\mathbb{B}_n) = \frac{n}{2}\log\left(\frac{2\pi e}{n}\right) + O(\log n)$, its leading term $-\frac{n}{2}\log n$ is vastly outpaced by the dominant $-\frac{1}{4}n^2 \log n$ term of $\log \operatorname{Vol}(\mathcal{E}_n)$.

Conversely, we can compare $\mathcal{E}_{n}$ to $\mathbb{B}_{d}$, as both can be viewed as subsets of the hypercube $[-1,1]^d$ where $d=n(n-1)/2$. In this parameterization, $\mathbb{B}_{d}$ decays to zero much faster than $\operatorname{Vol}(\mathcal{E}_{n})$. This follows from
\begin{align*}
\log \operatorname{Vol}(\mathcal{E}_n) &= \frac{d}{4}\log\left(\tfrac{2\pi^{2}e}{d}\right) +O\big(\sqrt{d}\big),\\
\log \operatorname{Vol}(\mathbb{B}_d) &= \frac{d}{2}\log\left(\tfrac{2\pi e}{d}\right)+O(\log d),
\end{align*}
so that for large $n$
$$
\log\frac{\operatorname{Vol}(\mathcal{E}_n)}{\operatorname{Vol}(\mathbb{B}_d)} = \frac{d}{4}\log\left(\tfrac{d}{2e}\right)+O(\sqrt{d}).
$$
Thus, $\operatorname{Vol}(\mathcal{E}_{n})$ exceeds $\operatorname{Vol}(\mathbb{B}_d)$ by the super-exponential factor $(d/2e)^{d/4}e^{O(\sqrt{d})}$ as $n$ (and consequently $d=n(n-1)/2$) increases; the $O(\sqrt{d})$ term has leading coefficient $-1/(2\sqrt{2})$, inherited from (\ref{eq:VolCnAsym}).

This difference in volume is a difference in width, coordinate by coordinate. Writing the volume of a set in $\mathbb{R}^d$ as $(2w)^d$ defines the half-width $w$ of the cube with the same volume. By (\ref{eq:VolCnAsym}), the elliptope has $w=(e\pi^2/4)^{1/4}n^{-1/2}(1+o(1))\approx1.61/\sqrt{n}$, whereas $\mathbb{B}_d$ has $w=(1+o(1))\sqrt{\pi e}/n\approx2.92/n$. The elliptope is therefore wider than the ball by the factor $(4e)^{-1/4}\sqrt{n}\approx0.55\sqrt{n}$ per coordinate, and since $n^2\sim2d$, this factor raised to the power $d$ reproduces the leading factor $(d/2e)^{d/4}$ above. Typical points show the same widths: an entry of a uniform draw from $\mathcal{E}_n$ has standard deviation $(n+1)^{-1/2}$, whereas a coordinate of a uniform point in $\mathbb{B}_d$ has standard deviation $(d+2)^{-1/2}\approx\sqrt{2}/n$. The extreme points of $\mathcal{E}_n$ include the $2^{n-1}$ rank-one cut matrices $C=xx^\prime$ with $x\in\{-1,1\}^n$, which sit exactly at vertices of the hypercube, but these spikes contribute almost nothing to the volume: by Proposition~\ref{prop:entrywise}, the matrices with an entry exceeding a fixed $r$ in absolute value make up an exponentially small fraction of it.

Both objects also exhibit boundary concentration, although in different metrics. If $X$ is uniformly distributed on $\mathbb{B}_d$, then its radial distance from the boundary satisfies
$$
d(1-\|X\|_2)\Rightarrow\operatorname{Exp}(1).
$$
For the elliptope, the corresponding boundary phenomenon is spectral rather than radial: under the uniform law on $\mathcal{E}_n$, the smallest eigenvalue satisfies
$$
d\lambda_{\min}(C_n)\Rightarrow\operatorname{Exp}(1),
$$
by Theorem~\ref{thm:specfloor} (see also Remark~\ref{rem:lambdamin}). Both boundary depths, scaled by the dimension count $d$, converge to a unit exponential; the results differ only in the notion of distance to the boundary, radial for the ball and spectral for the elliptope, so the analogy is close but not an equivalence.

The important distinction is shape. The unit ball is round: a typical point is close to the sphere, and all directions are equivalent. The elliptope is highly anisotropic. A typical draw satisfies $\|C_n-I_n\|_F\asymp\sqrt{n}$, whereas the rank-one cut matrices $C=xx^\prime$, $x\in\{-1,1\}^n$, lie at Frobenius distance $\sqrt{n(n-1)}\asymp n$ from the identity. Thus the elliptope has a substantial bulk near the identity on the entrywise scale, while simultaneously extending thin tentacles to selected corners of the hypercube.

\section{Proof of Proposition~\ref{prop:VolumeCn}}\label{sec:barnes}

This section proves Proposition~\ref{prop:VolumeCn}: first the recursion and the product formula (\ref{eq:VolCn}), then the asymptotic expansion (\ref{eq:VolCnAsym}) via the Barnes $G$-function. Throughout, $d=n(n-1)/2$.

In \citet[theorem~2.1]{JohnsonNaevdal:1998} it is shown that
$$
\operatorname{Vol}(\mathcal{E}_n)=2^d\prod_{j=2}^{n-1}(I_{j})^{j}=2\prod_{j=2}^{n-1}(2 I_{j})^{j},
$$
where
$$
I_{2k}=\tfrac{\pi}{2}\tfrac{1\cdot3\cdots(2k-1)}{2\cdot4\cdots(2k)},\qquad
I_{2k+1}=\tfrac{2\cdot4\cdots(2k)}{1\cdot3\cdots(2k+1)}.
$$

The expression in \citet{Joe:2006} is $\operatorname{Vol}(\mathcal{E}_n)=2^{\sum_{k=1}^{n-1}k^{2}}\prod_{k=1}^{n-1}\left[B\left(\tfrac{k+1}{2},\tfrac{k+1}{2}\right)\right]^{k}$, and the recursion in \citet{Joe:2006} is expressed as $\operatorname{Vol}(\mathcal{E}_n)=\operatorname{Vol}(\mathcal{E}_{n-1})2^{(n-1)(n-1)}\left[B\left(\tfrac{n}{2},\tfrac{n}{2}\right)\right]^{n-1}$.

\noindent{\bf Proof of Proposition~\ref{prop:VolumeCn}} 
Recall that
$\Gamma\left(m+\tfrac12\right)
   = \frac{(2m)!}{4^{m}m!}\sqrt{\pi}$,
and $\Gamma(\tfrac12)=\sqrt{\pi}$.
For $n=2k$ even, we have
$$
B\left(\tfrac{n+1}{2},\tfrac12\right)=B\left(k+\tfrac12,\tfrac12\right)
=\tfrac{\Gamma(k+\tfrac12)\Gamma(\tfrac12)}{\Gamma(k+1)}
 = \pi\tfrac{1\cdot 3\cdots(2k-1)}{2\cdot 4\cdots(2k)}
 = 2I_{2k}.
$$
and for $n=2k+1$ odd,
$$
 B\left(\tfrac{n+1}{2},\tfrac12\right)
 =B\left(k+1,\tfrac12\right)
  = \tfrac{\Gamma(k+1)\Gamma(\tfrac12)}{\Gamma(k+\tfrac32)} = 2\tfrac{2\cdot 4\cdots(2k)}{3\cdot 5\cdots(2k+1)}
 = 2I_{2k+1}.
$$
The recursive formula, 
$\operatorname{Vol}(\mathcal{E}_{n+1})=\operatorname{Vol}(\mathcal{E}_n)\left[B(\tfrac{n+1}{2},\tfrac{1}{2})\right]^n,
$ now follows, with $\operatorname{Vol}(\mathcal{E}_1)=1$ by definition ($\mathcal{E}_1=\{1\}$), so the recursion at $n=1$ gives $\operatorname{Vol}(\mathcal{E}_2)=B(1,\tfrac{1}{2})=2.$ Moreover, (\ref{eq:VolCn}) now follows by
$$
\operatorname{Vol}(\mathcal{E}_{n})=\prod_{k=1}^{n-1}
\left[B(\tfrac{k+1}{2},\tfrac{1}{2})\right]^k
= \prod_{k=1}^{n-1}\left(\tfrac{\Gamma(\tfrac{k+1}{2})\Gamma(\tfrac{1}{2})}{\Gamma(\tfrac{k+2}{2})}
\right)^k=\pi^{\frac{n(n-1)}{4}}\frac{\Gamma(1)\Gamma(\tfrac{3}{2})\cdots\Gamma(\tfrac{n}{2})}{\Gamma(\tfrac{n+1}{2})^{n-1}}.
$$
Equivalently, (\ref{eq:VolCn}) is $\Gamma_n(\tfrac{n+1}{2})/\Gamma(\tfrac{n+1}{2})^{n}$, where $\Gamma_n(a)=\pi^{n(n-1)/4}\prod_{i=1}^{n}\Gamma(a-\tfrac{i-1}{2})$ is the multivariate gamma function: the ratio of the Wishart and chi-squared normalizing constants in the Wishart/LKJ identity of Proposition~\ref{prop:WishartGram}(i) at $\nu=n+1$.

It remains to derive the asymptotic expansion (\ref{eq:VolCnAsym}) from this product formula, by writing $\prod_{k=1}^{n}\Gamma(\tfrac{k+1}{2})$ in terms of the Barnes $G$-function and applying the standard asymptotic expansions of $\log G$ and $\log\Gamma$.

The Barnes $G$ function, $G(z+1)=\Gamma(z)G(z)$ is a generalization of the superfactorial, $G(n+2)=1!2!\cdots n!$, for $n\in\mathbb{N}$, to the complex domain. Since 
$G(1)=1$ we have 
$$\prod_{k=1}^{n}\Gamma\big(\tfrac{k+1}{2}\big)
=\prod_{k=1}^{p}\Gamma(k) \times \prod_{k=1}^{q}\Gamma\left(k+\tfrac{1}{2}\right)
=G(p+1)\times \frac{G(q+\tfrac{3}{2})}{G(\tfrac{3}{2})},$$
where $p=\lfloor (n+1)/2\rfloor$ and $q=\lfloor n/2\rfloor$.
From $\operatorname{Vol}(\mathcal{E}_n)=\pi^{\frac{n(n-1)}{4}}\frac{\prod_{k=1}^{n}\Gamma\big(\tfrac{k+1}{2}\big)}{\Gamma\left(\frac{n+1}{2}\right)^{n}}$
we have
$$
\log \operatorname{Vol}(\mathcal{E}_n)=\tfrac{n(n-1)}{2}\log\sqrt{\pi}+\log G(p+1)+\log G\left(q+\tfrac{3}{2}\right)-\log G\left(\tfrac{3}{2}\right)-n\log\Gamma\left(\tfrac{n+1}{2}\right).
$$
From the asymptotic expansions of the logarithmic Barnes $G$ function and the logarithmic Gamma function, we have as $z\to\infty$,
\begin{eqnarray*}
    \log G(z+1)&=&\tfrac{1}{2}z^{2}\log z-\tfrac{3}{4}z^{2}+\tfrac{1}{2}z\log(2\pi)+\varsigma\log z+\varsigma^\prime+O\big(z^{-2}\big),
\\
\log\Gamma(z)&=&\big(z-\tfrac{1}{2}\big)\log z-z+\tfrac{1}{2}\log(2\pi)-\varsigma\tfrac{1}{z}+O\big(z^{-3}\big),
\end{eqnarray*}
where 
$\varsigma=\zeta(-1)=-\frac{1}{12}$
and
$\varsigma^\prime=\zeta^\prime(-1)\approx -0.1654$.

Consider the case where $n=2m$ is even. Then $p=q=m$ and
\begin{eqnarray*}
    \log G(m+1)
    &=&\tfrac{1}{2}m^{2}\log m-\tfrac{3}{4}m^{2}
+\tfrac{1}{2}m\log(2\pi)
+\varsigma\log m+\varsigma'
+O\big(m^{-2}\big)\\
\log G(m+\tfrac{3}{2})
&=&\tfrac{1}{2}(m+\tfrac{1}{2})^{2}\log(m+\tfrac{1}{2})
-\tfrac{3(m+\tfrac{1}{2})^{2}}{4}
+\tfrac{m+\tfrac{1}{2}}{2}\log(2\pi)\\
&&+\varsigma\log(m+\tfrac{1}{2})
+\varsigma^{\prime}
+O(m^{-2})\\
-2m\log\Gamma(m+\tfrac{1}{2})
&=&-2m^{2}\log(m+\tfrac{1}{2})+2m^{2}-m(\log(2\pi)-1)+\frac{4m\varsigma}{2m+1}
+O(m^{-2}).
\end{eqnarray*}
First, we consider the $m^{2}\log m$ terms and use
$$
\log(m+\tfrac{1}{2})=\log(m)+\tfrac{1}{m}\tfrac{1}{2}-\tfrac{1}{m^2}\tfrac{1}{8}+\tfrac{1}{m^3}\tfrac{1}{24}+O(m^{-4}),$$
which shows that
\begin{eqnarray*}
A_{2\log} &=&     
\tfrac{1}{2}m^{2}\log m+\tfrac{1}{2}(m+\tfrac{1}{2})^{2}\log(m+\tfrac{1}{2})-2m^{2}\log(m+\tfrac{1}{2})\\
&=&-\tfrac{1}{4}(2m^2-m)\log(m^2) +\tfrac{1}{8}\log(m)-\tfrac{3}{4}m+\tfrac{7}{16} +O(m^{-1})\\
&=& -\tfrac{1}{4}(2m^2-m)\log(m^2-\tfrac{1}{2}m) +\tfrac{1}{8}\log(m)- m+\tfrac{1}{2} +O(m^{-1})\\
&=& -\tfrac{d}{4}\log(\tfrac{d}{2})+\tfrac{1}{16}\log(m^2-\tfrac{1}{2}m)-m+\tfrac{1}{2} +O(m^{-1})\\
&=& -\tfrac{d}{4}\log(\tfrac{d}{2})+\tfrac{1}{16}\log(\tfrac{d}{2})-m+\tfrac{1}{2} +O(m^{-1}).
\end{eqnarray*}
Second, we consider the $m^2$ terms.
$$
A_{2} = -\tfrac{3}{4}m^{2}-\tfrac{3}{4}(m+\tfrac{1}{2})^{2}+2m^2
= \tfrac{1}{4}(2m^{2}-m)-\tfrac{1}{2}m -\tfrac{3}{16}=\tfrac{d}{4}-\tfrac{1}{2}m -\tfrac{3}{16}.
$$
Third, the $m$ terms, plus the remainders from $A_{2\log}$ and $A_2$
\begin{eqnarray*}
A_{1} &=& \tfrac{m}{2}\log(2\pi)+
(\tfrac{m}{2}+\tfrac{1}{4})\log(2\pi)
-m(\log(2\pi)-1)
-m+\tfrac{1}{2} -\tfrac{1}{2}m -\tfrac{3}{16}\\
&=& \tfrac{1}{4}\log(2\pi)-\tfrac{1}{2}m+\tfrac{5}{16}\\
&=& \tfrac{1}{4}\log(2\pi)-\tfrac{\sqrt{2d}}{4}+\tfrac{3}{16}+O(d^{-1/2})
\end{eqnarray*}
where we used 
$m=\tfrac{\sqrt{2d}}{2}+\tfrac{1}{4}+O(d^{-1/2}).$

Fourth, $\log(m)$ terms
\begin{eqnarray*}
A_{\log} &=& \varsigma\log(m)+\varsigma\log(m+\tfrac{1}{2})\\
&=& 2\varsigma\log(m)+O(m^{-1})=\varsigma\log(m^2-\tfrac{m}{2})+O(m^{-1})\\
&=&\varsigma\log(\tfrac{d}{2})+O(d^{-1/2}).
\end{eqnarray*}
Finally, the constant terms
$$
    A_0 = \varsigma^\prime+\varsigma^\prime +\frac{4m\varsigma}{2m+1}
    = 2(\varsigma^\prime+\varsigma)
+O(m^{-1}).
$$
Combining all terms, including $n(n-1)/2\log\sqrt\pi=\tfrac{d}{4}\log\pi^2$ and $-\log G(3/2)=\tfrac{\varsigma\log 2}{2}-\tfrac{\log\pi}{4}-\tfrac{3}{2}\varsigma^\prime$, we have
\begin{eqnarray*}
\log \operatorname{Vol}(\mathcal{E}_n) &=&\tfrac{d}{4}\log\pi^2-\tfrac{d}{4}\log(\tfrac{d}{2})+\tfrac{1}{16}\log(\tfrac{d}{2})
+\tfrac{d}{4}+\tfrac{\log(2\pi)}{4}-\tfrac{\sqrt{2d}}{4}+\tfrac{3}{16}+\varsigma\log(\tfrac{d}{2})\\
&& +2(\varsigma^\prime+\varsigma)-\tfrac{\log 2}{24}-\tfrac{\log\pi}{4}-\tfrac{3}{2}\varsigma^\prime
+O(d^{-1/2})\\
&=& \tfrac{d}{4}\log(\tfrac{2e\pi^2}{d})-\tfrac{\sqrt{2d}}{4}+(\varsigma+\tfrac{1}{16})\log(\tfrac{d}{2})-\tfrac{\log 2}{24}+\tfrac{\log(2)}{4}+\tfrac{3}{16}+\tfrac{1}{2}\varsigma^\prime+2\varsigma
+O(d^{-1/2})\\
&=& -\tfrac{d}{4}\log(\tfrac{d}{2e\pi^2})-\tfrac{\sqrt{d}}{2\sqrt{2}}-\tfrac{\log d}{48}+\tfrac{\log 2}{48}-\tfrac{\log 2}{24}+\tfrac{\log(2)}{4}+\tfrac{9}{48}+\tfrac{1}{2}\varsigma^\prime-\tfrac{8}{48}
+O(d^{-1/2})\\
&=& -\tfrac{d}{4}\log(\tfrac{d}{2e\pi^2})-\tfrac{1}{2\sqrt{2}}\sqrt{d}-\tfrac{1}{48}\log d +\tfrac{(1+\varsigma)\log 2-\varsigma}{4}
+\tfrac{1}{2}\varsigma^\prime
+O(d^{-1/2}),
\end{eqnarray*}
where the leading term, $-\frac{d}{4}\log d$, can be expressed as $-\frac{n^2}{4}\log n$, using $d\sim n^2/2$.

It remains to extend the expansion from even to odd $n$. Let $F(d)$ denote the right-hand side of (\ref{eq:VolCnAsym}) without the remainder term, and let $d_n=n(n-1)/2$, so that $\log\operatorname{Vol}(\mathcal{E}_n)=F(d_n)+O(d_n^{-1/2})$ for $n$ even. The recursion in Proposition~\ref{prop:VolumeCn} gives $\log\operatorname{Vol}(\mathcal{E}_{n+1})=\log\operatorname{Vol}(\mathcal{E}_n)+n\log B(\tfrac{n+1}{2},\tfrac{1}{2})$, where $d_{n+1}=d_n+n$. From $\Gamma(z+\tfrac{1}{2})/\Gamma(z)=z^{1/2}(1-\tfrac{1}{8z}+O(z^{-2}))$ with $z=\tfrac{n+1}{2}$ we obtain
$$
n\log B(\tfrac{n+1}{2},\tfrac{1}{2})=\tfrac{n}{2}\log(2\pi)-\tfrac{n}{2}\log(n+1)+\tfrac{1}{4}+O(n^{-1}),
$$
while a direct expansion of $F$ yields
$$
F(d_n+n)-F(d_n)=\tfrac{n}{2}\log(2\pi)-\tfrac{n}{4}\log n-\tfrac{n}{4}\log(n+1)+O(n^{-1}).
$$
The difference between the two right-hand sides is $\tfrac{n}{4}\log\tfrac{n+1}{n}-\tfrac{1}{4}+O(\tfrac{1}{n})=O(\tfrac{1}{n})$, so $\log\operatorname{Vol}(\mathcal{E}_{n+1})=F(d_{n+1})+O(d_{n+1}^{-1/2})$, and the expansion (\ref{eq:VolCnAsym}) holds for all $n$.
\hfill$\square$

\section{Proofs of Auxiliary Results}\label{sec:auxproofs}

\noindent{}{\bf Proof of Proposition~\ref{prop:entrywise}.}
Under the uniform law, $\operatorname{Vol}(\mathcal{E}_n(r))/\operatorname{Vol}(\mathcal{E}_n)=\Pr(\max_{i<j}|C_{ij}|\leq r)$, so $1-\operatorname{Vol}(\mathcal{E}_n(r))/\operatorname{Vol}(\mathcal{E}_n)=\Pr(\max_{i<j}|C_{ij}|>r)$. Writing $p_n(r)=\Pr(|C_{12}|>r)$ and $d=\binom{n}{2}$, the single-event lower bound and union bound give
$$p_n(r)\leq\Pr\left(\max_{i<j}|C_{ij}|>r\right)\leq d\cdot p_n(r),$$
so it suffices to show $\log p_n(r)=\tfrac{n}{2}\log(1-r^2)+O(\log n)$.

Since $(C_{12}+1)/2\sim\operatorname{Beta}(n/2,n/2)$, the density of $C_{12}$ on $(-1,1)$ is
$$
f_n(x)=Z_n(1-x^2)^{n/2-1},\qquad
Z_n=\frac{\Gamma(\tfrac{n+1}{2})}{\sqrt{\pi}\Gamma(\tfrac{n}{2})}=\frac{2^{1-n}}{B(\tfrac{n}{2},\tfrac{n}{2})},
$$
where the second expression for $Z_n$ follows from the Legendre duplication formula; note that $Z_n$, and not $B(\tfrac{n}{2},\tfrac{n}{2})^{-1}$, is the relevant normalizing constant, the two differing by the factor $2^{n-1}$. By symmetry of $f_n$,
$$p_n(r) = 2Z_n\int_r^1(1-x^2)^{n/2-1}dx.$$
By Stirling, $\Gamma(\tfrac{n+1}{2})/\Gamma(\tfrac{n}{2})\sim\sqrt{n/2}$, so $Z_n\sim\sqrt{n/(2\pi)}$ and $\log Z_n=\tfrac{1}{2}\log n+O(1)$. It remains to show that the core integral satisfies $\log\int_r^1(1-x^2)^{n/2-1}dx=\tfrac{n}{2}\log(1-r^2)+O(\log n)$.

\textit{Upper bound.} Since $x\mapsto\log(1-x^2)$ is strictly concave, its tangent at $r$ is a global upper bound: $(1-x^2)\leq(1-r^2)e^{-2r(x-r)/(1-r^2)}$ for all $x\geq r$. Hence
$$\int_r^1(1-x^2)^{n/2-1}dx\leq(1-r^2)^{n/2-1}\int_0^\infty e^{-(n-2)rt/(1-r^2)}dt=\frac{(1-r^2)^{n/2}}{(n-2)r}.$$

\textit{Lower bound.} Set $\ell=(1-r^2)/[(n-2)r]$, the reciprocal of the exponential-decay rate appearing in the upper bound. Since $(1-x^2)^{n/2-1}$ is decreasing, restricting to $[r,r+\ell]$ gives
$$\int_r^1(1-x^2)^{n/2-1}dx\geq\ell\cdot(1-(r+\ell)^2)^{n/2-1}.$$
Since $2r\ell/(1-r^2)=2/(n-2)$ and $\ell^2=O(n^{-2})$ for fixed $r\in(0,1)$,
$$
\frac{1-(r+\ell)^2}{1-r^2}
=1-\frac{2r\ell+\ell^2}{1-r^2}
=1-\frac{2}{n-2}+O(n^{-2}),
$$
and taking logarithms gives $(\tfrac{n}{2}-1)\log\bigl(1-\tfrac{2}{n-2}+O(n^{-2})\bigr)=-1+O(n^{-1})$. Hence
$$\left(\frac{1-(r+\ell)^2}{1-r^2}\right)^{n/2-1}=e^{-1}(1+o(1)),$$
so that $(1-(r+\ell)^2)^{n/2-1}\geq\frac{1}{2e}(1-r^2)^{n/2-1}$ for all large $n$, giving
$$\int_r^1(1-x^2)^{n/2-1}dx\geq\frac{(1-r^2)^{n/2}}{2e(n-2)r}.$$

Combining the upper and lower bounds and using $\log((n-2)r)=\log n+O(1)$,
$$\log\int_r^1(1-x^2)^{n/2-1}dx=\frac{n}{2}\log(1-r^2)+O(\log n).$$
Adding $\log(2Z_n)=\tfrac{1}{2}\log n+O(1)=O(\log n)$ and substituting into the two-sided bound above with $\log d=2\log n+O(1)$ gives $\log\Pr(\max_{i<j}|C_{ij}|>r)=\tfrac{n}{2}\log(1-r^2)+O(\log n)=-\tfrac{n}{2}|\log(1-r^2)|+O(\log n)$, completing the proof.\hfill$\square$

\noindent{}{\bf Proof of Proposition~\ref{prop:C-I}.}
This is the case $\eta_n=1$, $\nu_n=n+1$, of Theorem~\ref{thm:LKJscaling}(ii), proved in the Appendix.\hfill$\square$

\noindent{}{\bf Proof of Proposition~\ref{prop:WishartGram}.}
We give the short density calculation for completeness.
(i) The Wishart distribution $W_n(\nu,I_n)$ has density proportional to $\det(W)^{(\nu-n-1)/2}e^{-\operatorname{tr}(W)/2}$ on the positive definite matrices. \citet[theorem~3.2.1]{Muirhead:1982} derives this for integer $\nu\geq n$ and observes in the discussion that follows that the expression remains a density for every real $\nu>n-1$, which is the form used here. Write $w=(w_1,\ldots,w_n)$ for the diagonal of $W$, so that $D=\operatorname{diag}(w_1^{1/2},\ldots,w_n^{1/2})$. The change of variables $W\mapsto(w,C)$, where $w_i=W_{ii}>0$ and $W_{ij}=\sqrt{w_iw_j}C_{ij}$ for $i<j$, has Jacobian $\prod_{i=1}^n w_i^{(n-1)/2}$. Using $\det(W)=\det(C)\prod_{i=1}^n w_i$, the joint density of $(w,C)$ is proportional to
$$
\det(C)^{(\nu-n-1)/2}\times\prod_{i=1}^n w_i^{(\nu-2)/2}e^{-w_i/2},
$$
which factorizes. Hence $C$ is independent of the diagonal scale $w$ with density proportional to $\det(C)^{(\nu-n-1)/2}$ on $\mathcal{E}_n$, which is the LKJ$(\eta)$ density with $\nu=n+2\eta-1$. The exponent vanishes if and only if $\nu=n+1$, and because $\det(C)$ is not constant on $\mathcal{E}_n$ (it equals one at $I_n$ and vanishes at the boundary), no other value of $\nu$ yields a uniform density.

(ii) For integer $\nu=k$, $W$ is distributed as the Gram matrix $(z_i^\prime z_j)_{i,j=1}^n$ of independent $z_1,\ldots,z_n\sim N(0,I_k)$. In this case $C=V^\prime V$ with $v_i=z_i/\|z_i\|$, and $v_1,\ldots,v_n$ are independent and uniformly distributed on the unit sphere in $\mathbb{R}^k$. Part (ii) is standard; \citet{Joe:2006} gives this Gram construction explicitly in the random-correlation context, and \citet{LewandowskiKurowickaJoe:2009} use the same spherical representation.\hfill$\square$

\noindent{}{\bf Proof of Theorem~\ref{thm:MP}.}
This is the case $\eta_n=1$ of Theorem~\ref{thm:LKJscaling}(i), proved in the Appendix. It also follows directly from the Marchenko--Pastur theorem for sample correlation matrices, \citet[theorem~1.2]{Jiang:2004} in the centered form with $T=n+2$ and \citet[theorem~2.2]{Heiny:2022} in the uncentered form with $T=n+1$, via the representation of Proposition~\ref{prop:WishartGram}(ii); almost-sure convergence holds under the coupling of Remark~\ref{rem:coupling}.\hfill$\square$

\noindent{}{\bf Proof of Proposition~\ref{prop:precision}.}
(i) Since $\operatorname{tr}(C_n^{-1})\geq1/\lambda_{\min}(C_n)$, we have $n^2/\operatorname{tr}(C_n^{-1})\leq n^2\lambda_{\min}(C_n)=\tfrac{n^2}{d}\cdot d\lambda_{\min}(C_n)$, which is $O_p(1)$ because $n^2/d\rightarrow2$ and $d\lambda_{\min}(C_n)$ converges in distribution to a unit exponential by Theorem~\ref{thm:specfloor}. Hence $\tfrac{1}{n}\operatorname{tr}(C_n^{-1})=n\bigl(n^2/\operatorname{tr}(C_n^{-1})\bigr)^{-1}\rightarrow\infty$ in probability.

(ii) By \citet[theorem~1.2]{Jiang:2004}, the empirical spectral distribution $F_n$ of $\hat{C}_n$ converges weakly to $F_{\mathrm{MP},c}$ almost surely. For $c<1$ the Marchenko--Pastur support is $[(1-\sqrt{c})^2,(1+\sqrt{c})^2]$ with $(1-\sqrt{c})^2>0$. The extreme eigenvalues of $\hat{C}_n$ converge almost surely to these edges: \citet[theorem~3.3]{HeinyMikosch:2018} give $\lambda_{\min}(\hat{C}_n)\to(1-\sqrt{c})^2$ and $\lambda_{\max}(\hat{C}_n)\to(1+\sqrt{c})^2$ almost surely for the sample correlation matrix itself, under a finite fourth moment. Their result is stated for the uncentered matrix, which for Gaussian observations is the centered one with $T-1$ in place of $T$, so the aspect ratio and the limits are unchanged. Because the spectrum of $\hat{C}_n$ is eventually bounded away from zero, $x\mapsto x^{-1}$ is eventually bounded and continuous on the support, and weak convergence of $F_n$ gives $\tfrac{1}{n}\operatorname{tr}(\hat{C}_n^{-1})=\int x^{-1}dF_n(x)\to\int x^{-1}dF_{\mathrm{MP},c}(x)=\tfrac{1}{1-c}$ a.s., where the inverse moment is a standard computation with the density $f_{\mathrm{MP},c}$.\hfill$\square$

The finite-$n$ identity quoted after Proposition~\ref{prop:precision} follows from the Bartlett decomposition with variable $j$ placed last: writing $C=D^{-1}WD^{-1}$ with $W\sim W_n(\nu,I_n)$ and $\nu=n+2\eta-1$, one has $(C^{-1})_{jj}=W_{jj}(W^{-1})_{jj}=W_{jj}/W_{jj\cdot}$, where the residual sum of squares $W_{jj\cdot}$ of variable $j$ regressed on the others and the explained sum of squares $W_{jj}-W_{jj\cdot}$ are independent $\chi^2_{\nu-n+1}=\chi^2_{2\eta}$ and $\chi^2_{n-1}$ variables \citep[theorem~3.2.10]{Muirhead:1982}. Hence $1/(C^{-1})_{jj}=W_{jj\cdot}/W_{jj}\sim\operatorname{Beta}(\eta,\tfrac{n-1}{2})$, and $\mathbb{E}[(C^{-1})_{jj}]=(\eta+\tfrac{n-1}{2}-1)/(\eta-1)=1+\tfrac{n-1}{2(\eta-1)}$ for $\eta>1$, while the expectation is infinite for $0<\eta\leq1$.

\noindent{}{\bf Proof of Proposition~\ref{prop:identitycenter}.}
Identify a symmetric matrix with unit diagonal with its off-diagonal vector in $\mathbb{R}^d$, let $f$ be the density of $E$, and set $v(C)=\Pr(C+E\in\mathcal{E}_n)=\int\mathbf{1}_{\mathcal{E}_n}(C+e)f(e)de$. This is the convolution of the log-concave functions $\mathbf{1}_{\mathcal{E}_n}$ and $e\mapsto f(-e)$, hence log-concave \citep{Prekopa:1973}. For every sign matrix $S$, the map $C\mapsto SCS$ maps $\mathcal{E}_n$ onto itself and preserves the law of $E$, so $v(SC_0S)=v(C_0)$. Since $\mathbb{E}[s_is_j]=0$ for $i\neq j$ when the signs are independent and uniform, $2^{-n}\sum_SSC_0S=I_n$, and log-concavity gives $v(I_n)\geq\prod_Sv(SC_0S)^{2^{-n}}=v(C_0)$. For independent entries uniform on $[-r,r]$, $\Pr(I_n+E\in\mathcal{E}_n)=\operatorname{Vol}(\{C\in\mathcal{E}_n:\max_{i<j}|C_{ij}|\leq r\})/(2r)^{d}=\operatorname{Vol}(\mathcal{E}_n(r))/(2r)^{d}$. The Gershgorin claim in the text follows from $\lambda_{\min}(I_n+H)\geq1-(n-1)r$ for $|H_{ij}|\leq r$, with equality when every $H_{ij}=-r$.\hfill$\square$

\section{Exact Volumes and Recursions}\label{sec:exact}

\begin{figure}[htbp!]
\centering{}\includegraphics[width=0.48\textwidth]{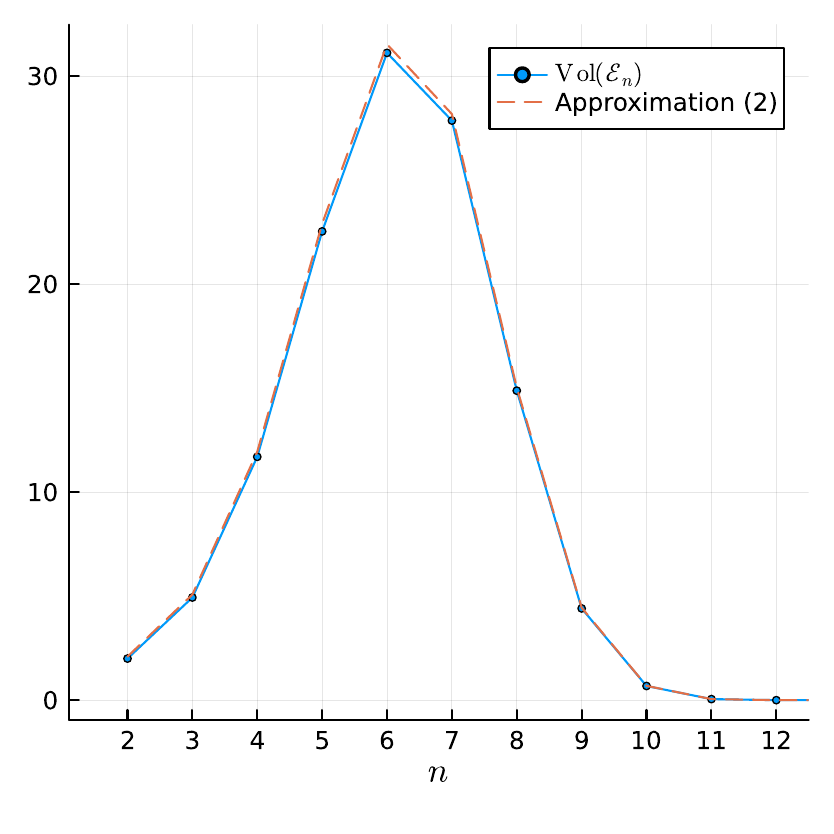}\includegraphics[width=0.48\textwidth]{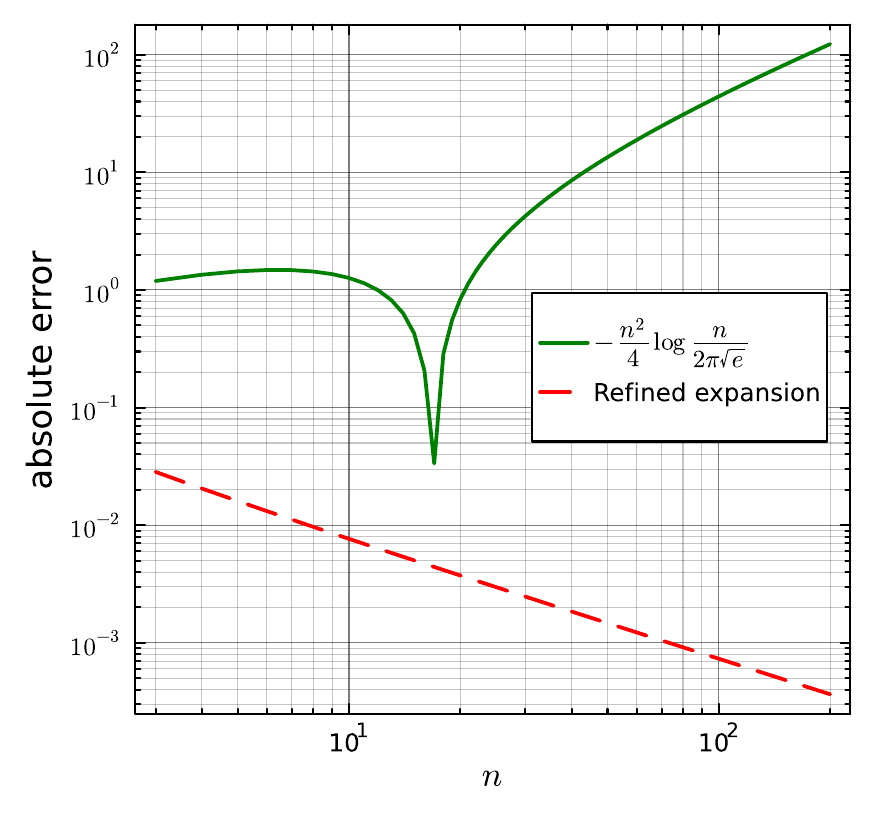}
\caption{{\small{}\textit{Left:} the volume of the elliptope $\mathcal{E}_n$ for small dimensions, exact and as approximated by (\ref{eq:VolCnAsym}). \textit{Right:} the absolute approximation error for $\log\operatorname{Vol}(\mathcal{E}_n)$ on a log-log scale, for the leading-order expression $-\tfrac{n^2}{4}\log\tfrac{n}{2\pi\sqrt{e}}$ and for the refined expansion (\ref{eq:VolCnAsym}). The leading-order error grows without bound (the dip near $n\approx17$ is a sign change), whereas the refined expansion has error below $0.02$ for all $n\geq5$, decaying at the rate $O(d^{-1/2})$.\label{fig:VolCorr}}}
\end{figure}

\begin{table}[htbp!]
    \centering
    \begin{scriptsize}
    \setlength{\tabcolsep}{3pt}
    \begin{tabularx}{\textwidth}{p{1mm}p{5mm}p{7mm}p{1mm}>{\hsize=1.9\hsize}Y>{\hsize=0.55\hsize}Xp{1mm}p{1mm}>{\hsize=0.55\hsize}X}
    \toprule 
    &$n$ & $d(n)$ && \multicolumn{2}{c}{$\operatorname{Vol}(\mathcal{E}_n)$} &&& $\operatorname{Vol}(\mathcal{E}_n)/2^d$ \\ 
    \cmidrule{1-3}  \cmidrule{5-6}  \cmidrule{8-9}  
    & 2  & 1 &&2& \num{2.000} &&& 1.000 \\[2mm] 
    & 3  & 3 &&$\frac{1}{2}\pi^{2}$& \num[scientific-notation = fixed]{4.935}                               &&& \num[scientific-notation=fixed]{0.617} \\[2mm] 
    & 4  & 6 &&$\tfrac{2^{5}}{3^{3}}\pi^{2}$& \num[scientific-notation=fixed]{11.697}                     &&& \num[scientific-notation =fixed]{0.183} \\[2mm] 
    & 5  & 10 &&$\tfrac{3}{2^{7}}\pi^{6}$& 22.533 &&& \num{2.20e-2} \\[2mm]
    & 6  & 15 &&$\tfrac{2^{13}}{3^{4}5^{5}}\pi^{6}$ & 31.114 &&& \num{9.50e-4} \\[2mm]
    & 7 & 21 &&$\tfrac{5}{2^{11}3^{4}}\pi^{12}$ & 27.859 &&& \num{1.33e-5} \\[2mm]
    & 8 & 28 &&$\tfrac{2^{24}}{3^{4}5^{6}7^{7}}\pi^{12}$                                                    & 14.877 &&& \num{5.54e-8} \\[2mm] 
    & 9  & 36 && $\tfrac{5^{2}7}{2^{32}3^{4}}\pi^{20}$ & 4.412 &&& \num{6.42e-11} \\[2mm] 
    & 10 & 45 &&$\tfrac{2^{40}}{3^{22}5^{7}7^{8}}\pi^{20}$ & 0.682 &&& \num{1.94e-14} \\[2mm]
    & 11 & 55 & &$\frac{7^{2}}{2^{40}3^{2}5^{7}}\pi^{30}$& 0.0521 &&& \num{1.44e-18} \\[2mm] 
    & 12 & 66 &&$\tfrac{2^{59}}{3^{24}5^{7}7^{9}11^{11}}\pi^{30}$& 0.00186 &&& \num{2.53e-23} \\[2mm] 
    & 13 & 78 &&$\tfrac{11^1 7^{3}}{2^{61}3^{12}5^{7}}\pi^{42}$& \num{2.99e-5} &&& \num{9.90e-29} \\[2mm] 
    & 14 & 91 &&$\tfrac{2^{82}}{3^{25}5^{7}7^{10}11^{12}13^{13}}\pi^{42}$& \num{2.07e-7} &&& \num{8.34e-35} \\[2mm] 
    & 15 & 105 &&$\frac{11^2 13}{2^{72}3^{11}5^{7}7^{10}}\pi^{56}$& \num{5.90e-10} &&& \num{1.45e-41}\\[2mm] 
    & 16 & 120 &&$\tfrac{2^{108}}{3^{41}5^{22}7^{10}11^{13}13^{14}}\pi^{56}$& \num{6.73e-13} &&& \num{5.06e-49} \\[2mm] 
    & 18 & 153 &&$\tfrac{2^{140}}{3^{43}5^{23}7^{10}11^{14}13^{15}17^{17}}\pi^{72}$& \num{4.89e-20}           &&& \num{4.28e-66} \\[2mm]
    & 20 & 190 &&$\frac{2^{175}}{3^{43}5^{24}7^{10}11^{15}13^{16}17^{18}19^{19}}\pi^{90}$& \num{6.21e-29} &&& \num{3.96e-86} \\[2mm] 
    & 30 & 435 && $\tfrac{2^{407}}{   3^{149}5^{78}7^{34}11^{16}13^{19}17^{23}19^{24}23^{26}29^{29}
  }\pi^{210}$& \num{8.80e-103} &&& \num{9.92e-234}\\[2mm]
    \bottomrule \end{tabularx}   
    \end{scriptsize}
    \caption{The table reports the intrinsic dimension $d(n)=n(n-1)/2$, the closed-form expression for $\operatorname{Vol}(\mathcal{E}_n)$, its numerical value, and its volume relative to the enclosing hypercube, $\operatorname{Vol}(\mathcal{E}_n)/2^d$, for $n\in \{2,3,\ldots,16,18,20,30\}$.\label{tab:CnVolume}}
\end{table}

The recursions that increase the dimension by two are relatively simple, but differ for even and odd dimensions.
\begin{lemma}[Two-step recursions]\label{lem:2step}Suppose that $n=2k$ is even. Then
\begin{equation}    
\frac{\operatorname{Vol}(\mathcal{E}_{n+1})}{\operatorname{Vol}(\mathcal{E}_{n-1})  }
        =    
            \pi^n\frac{k}{n^{n}}\binom{n}{k}  
      \quad\text{and}\quad
     \frac{\operatorname{Vol}(\mathcal{E}_{n+2})}{\operatorname{Vol}(\mathcal{E}_n) }
     =   \pi^n \frac{2^{2n+1}}{(n+1)^{n+1}}\frac{1}{\binom{n}{k}}.   \label{eq:Skip2}
\end{equation}
\end{lemma}
Combining the {\em even} and {\em odd} two-step recursions in (\ref{eq:Skip2})
 leads to the following exact recursion for the log-volume, which has computationally simple increments.
\begin{proposition}[Exact recursion]\label{prop:exact}Let $L_n\equiv\log\operatorname{Vol}(\mathcal{E}_{n})$, so that $L_2=\log 2$, and set $L_1=L_0\equiv 0$. Then 
$$
L_{n+2}=-L_{n+1}+L_{n}+L_{n-1}+\Delta_n,\quad n\geq 1,
$$
where $\Delta_n=2n\log(2\pi)-(n-1)\log(n)-(n+1)\log(n+1)$.
\end{proposition}

\bigskip
\noindent{}{\bf Proof of Lemma~\ref{lem:2step}.}
If $n$ is even, then
\begin{eqnarray*}
    \frac{\operatorname{Vol}(\mathcal{E}_{n+1})}{\operatorname{Vol}(\mathcal{E}_{n-1})  }
        &=& \left(B(\tfrac{n+1}{2},\tfrac{1}{2})\right)^n
            \left(B(\tfrac{n}{2},\tfrac{1}{2})\right)^{n-1}\\
        &=& \left(\tfrac{\Gamma(k+\tfrac{1}{2})\Gamma(\tfrac{1}{2})}{\Gamma(k+1)}    \right)^n 
        \left(\tfrac{\Gamma(k)\Gamma(\tfrac{1}{2})}{\Gamma(k+\tfrac{1}{2})}    \right)^{n-1}\\
        &=&{\pi^{n-1/2}}\frac{\Gamma(k+\tfrac{1}{2})}{k!}
        k^{-n+1} \\
        &=&{\pi^{n}}\frac{2^{-n}n!}{(k!)^2}
        k^{-n+1} \\
        &=&{\pi^{n}}\frac{k}{n^n}\binom{n}{k}
    \end{eqnarray*}
    where we used the duplication formula, $\Gamma(k+\tfrac{1}{2})=2^{-2k}\sqrt{\pi}(2k)!/k!$, in the second-to-last equality.
    Next
    \begin{eqnarray*}        
    \frac{\operatorname{Vol}(\mathcal{E}_{n+2})}{\operatorname{Vol}(\mathcal{E}_{n})  }
        &=& \left(B(\tfrac{n+2}{2},\tfrac{1}{2})\right)^{n+1}
            \left(B(\tfrac{n+1}{2},\tfrac{1}{2})\right)^n\\
        &=& \left(\tfrac{\Gamma(k+1)\Gamma(\tfrac{1}{2})}{\Gamma(k+\tfrac{3}{2})}    \right)^{n+1} 
        \left(\tfrac{\Gamma(k+\tfrac{1}{2})\Gamma(\tfrac{1}{2})}{\Gamma(k+1)}    \right)^n \\
        &=&\pi^{n+1/2}\frac{\Gamma(k+1)}{(k+\tfrac{1}{2})^{n+1} \Gamma(k+\tfrac{1}{2}) }    \\
        &=&\pi^{n}\frac{2^{n+1}}{(n+1)^{n+1}}\frac{\sqrt{\pi}\Gamma(k+1)^2}{ \Gamma(k+\tfrac{1}{2})\Gamma(k+1) }  \\
        &=&\pi^{n}\frac{2^{2n+1}}{(n+1)^{n+1}}\frac{(k!)^2}{ (2k)! }\\
        &=&\pi^{n}\frac{2^{2n+1}}{(n+1)^{n+1}}\frac{1}{ \binom{n}{k}},
    \end{eqnarray*}
   where we used the same identity in the second-to-last equality.\hfill$\square$

\noindent{}{\bf Proof of Proposition~\ref{prop:exact}.}
 For $n=2k$ even, Lemma~\ref{lem:2step} gives directly
$$
\frac{\operatorname{Vol}(\mathcal{E}_{n+1})}{\operatorname{Vol}(\mathcal{E}_{n-1})  }
\frac{\operatorname{Vol}(\mathcal{E}_{n+2})}{\operatorname{Vol}(\mathcal{E}_n) } = \pi^{2n}
\frac{k2^{2n+1}}{n^{n}(n+1)^{n+1}}
=  \frac{(2\pi)^{2n}}{ n^{n-1}(n+1)^{n+1}}.
$$
For $n=2k+1$ odd, the same product equals $\frac{\operatorname{Vol}(\mathcal{E}_{2k+2})}{\operatorname{Vol}(\mathcal{E}_{2k})} \frac{\operatorname{Vol}(\mathcal{E}_{2k+3})}{\operatorname{Vol}(\mathcal{E}_{2k+1})}$; this is the second formula of Lemma~\ref{lem:2step} evaluated at $2k$ times the first formula evaluated at $2k+2$, which simplifies to the same expression $\frac{(2\pi)^{2n}}{n^{n-1}(n+1)^{n+1}}$. Taking logs and rearranging gives the result in both cases.\hfill$\square$

\section{Illustration of the Point-Process Limit}\label{sec:pppfig}

Figure~\ref{fig:PPP} illustrates Theorem~\ref{thm:PPP}. In the left panel each step function is a single realization of $\Xi_n(t,\infty)$, the number of centered squared correlations exceeding $t$. The realizations fluctuate around the tail mean measure $\Lambda(t,\infty)$, and the fluctuations do not vanish as $n$ grows, since the limit is a Poisson process. The right panel shows the exact pre-limit means $\lambda_n(t,\infty)$ approaching $\Lambda(t,\infty)$.

\begin{figure}[htbp!]
\centering{}\includegraphics[width=0.8\textwidth]{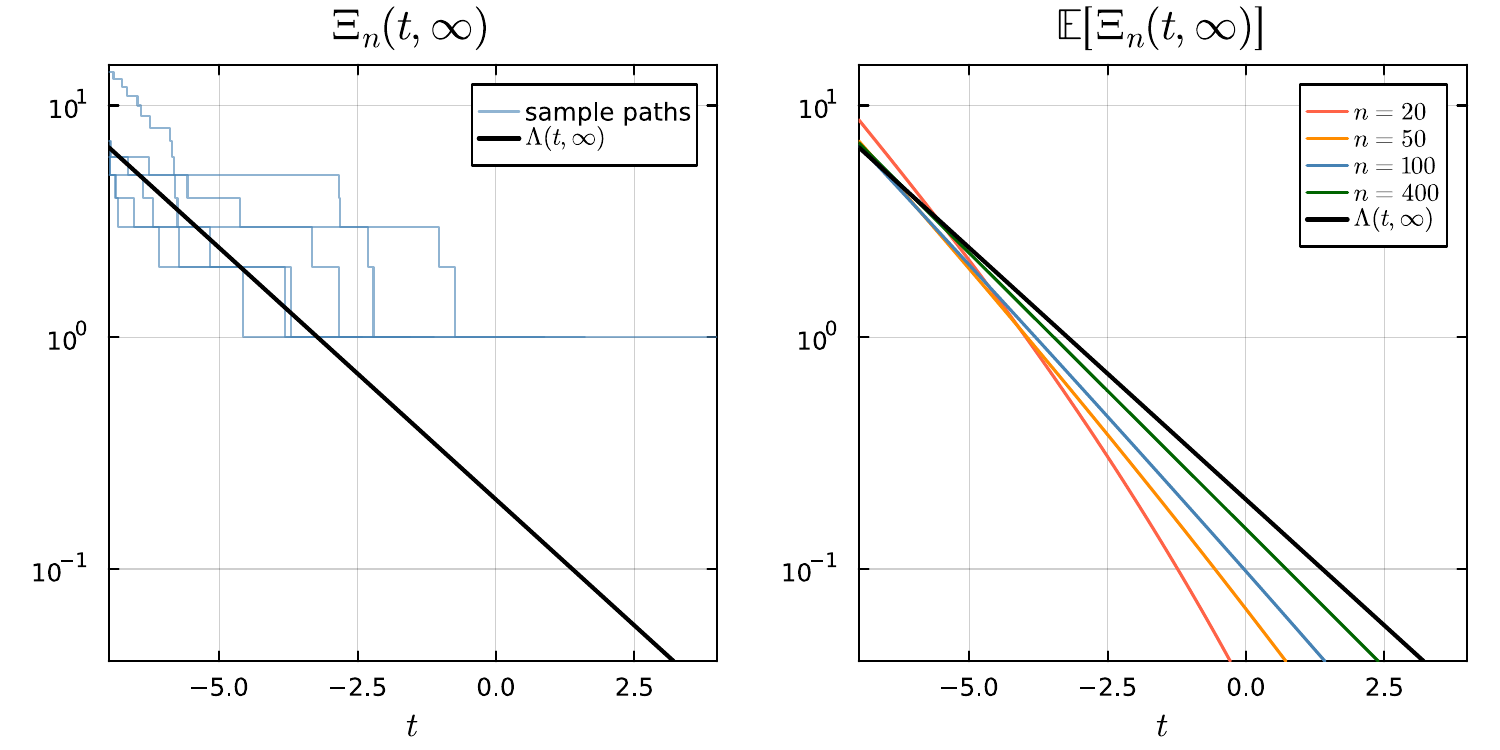}
\caption{{\small{}The exceedance counting process $\Xi_n(t,\infty)=\#\{i<j:(n+1)C_{ij}^2-\mu_n>t\}$ for a uniformly distributed correlation matrix $C\in\mathcal{E}_n$, plotted as a decreasing step function of $t$. The smooth curve is the tail mean measure $\Lambda(t,\infty)=(8\pi)^{-1/2}e^{-t/2}$ of the limiting $\operatorname{PPP}(\Lambda)$. \textit{Left:} six independent realizations at $n=1500$, illustrating the Poisson character of the variability. \textit{Right:} the exact pre-limit mean $\lambda_n(t,\infty)=\mathbb{E}[\Xi_n(t,\infty)]=d\cdot\Pr\{(n+1)C_{12}^2-\mu_n>t\}$, computed from the Beta marginal, for $n=20,50,100,400$. The pre-limit mean measure converges slowly, with logarithmic corrections under this centering, so a visible gap remains at $n=400$; this is why the total-variation bound of Theorem~\ref{thm:PPP} is stated relative to the exact finite-$n$ means $\lambda_n(B_r)$ rather than to $\Lambda$.\label{fig:PPP}}}
\end{figure}

\section{Triangle Counts at Moderate Thresholds}\label{sec:triangles}

Corollary~\ref{cor:ppp_consequences} concerns the extreme thresholds $\mathbb{E}N_n(\tau_n)\rightarrow\lambda$, at which triangles vanish under the uniform law and under the Erd\H{o}s--R\'enyi benchmark alike. At the moderate thresholds used in applied work the two nulls differ, and the difference is computable exactly. Under the uniform law on $\mathcal{E}_n$, every $3\times3$ principal submatrix $C_3$ is the Gram matrix of three independent uniform vectors on the unit sphere in $\mathbb{R}^{n+1}$ (Proposition~\ref{prop:WishartGram}(ii)), so its density is proportional to $\det(C_3)^{(n-3)/2}$ on $\mathcal{E}_3$, that is, $C_3\sim\operatorname{LKJ}_3(\tfrac{n-1}{2})$. In the vine parameterization of \citet{LewandowskiKurowickaJoe:2009}, $a=C_{12}$ and $b=C_{13}$ are independent with $(a+1)/2\sim\operatorname{Beta}(\tfrac{n}{2},\tfrac{n}{2})$, the partial correlation $\varrho=\rho_{23;1}$ is independent of both with $(\varrho+1)/2\sim\operatorname{Beta}(\tfrac{n-1}{2},\tfrac{n-1}{2})$, and $c=C_{23}=ab+\varrho\sqrt{(1-a^2)(1-b^2)}$. The triangle probability $q_n(\tau)=\Pr(|a|,|b|,|c|>\tau)$ is therefore a two-dimensional integral of Beta densities against the Beta distribution function of $\varrho$, and the expected number of triangles is $\binom{n}{3}q_n(\tau)$, against $\binom{n}{3}p_n(\tau)^3$ under $G(n,p_n(\tau))$. Table~\ref{tab:triangles} reports the ratio $q_n(\tau)/p_n(\tau)^3$ at thresholds matched to a given edge density, computed by Gauss--Legendre quadrature (\texttt{TriangleInflation.jl} in the replication files). For $T$ centered Gaussian observations the same computation applies with $n+1$ replaced by the spherical dimension $T-1$, the minor then being $\operatorname{LKJ}_3(\tfrac{T-3}{2})$.

The inflation has an explicit expansion. Fix $p\in(0,1)$, choose the thresholds $\tau_n$ so that $p_n(\tau_n)=p$, put $k=n+1$, and scale the three correlations by $\sqrt{k}$. If $g_k^{(3)}$ is their joint density and $g_k$ the common marginal density, then
$$
\frac{g_k^{(3)}(x,y,z)}{g_k(x)g_k(y)g_k(z)}=1+\frac{xyz}{\sqrt{k}}+\frac{(x^2-1)(y^2-1)(z^2-1)}{2k}+R_k(x,y,z),
$$
from the expansion $\log\det(C_3)=-(a^2+b^2+c^2)+2abc+O(a^4+b^4+c^4)$ together with the expansions of the three marginal densities. On the box $\max(|x|,|y|,|z|)\leq L_k=\sqrt{A\log k}$ these Taylor expansions hold uniformly, because $L_k/\sqrt{k}\rightarrow0$, and give $|R_k|\leq Ck^{-3/2}(1+L_k)^{m}$ for constants $C$ and $m$. Outside the box, the Beta tail bound in the proof of Lemma~\ref{lem:betatail} gives $\Pr(\sqrt{k}|C_{12}|>L_k)=O(k^{-A/2})$. For $A>2$ the contributions from outside the box are therefore $o(k^{-1})$, both for the joint density and, by the Cauchy--Schwarz inequality, for the polynomial terms of the expansion. On the exceedance event the odd term integrates to zero by symmetry, and since $\mathbb{E}[(Z^2-1)\mathbf{1}\{|Z|>z\}]=2z\phi(z)$ for $Z\sim N(0,1)$, the even term gives, with $z=\Phi^{-1}(1-p/2)$,
$$
\frac{q_n(\tau_n)}{p^3}=1+\frac{1}{2(n+1)}\Bigl(\frac{2z\phi(z)}{p}\Bigr)^{3}+o(n^{-1}).
$$ The expansion column of Table~\ref{tab:triangles} is within about 3 percent of the exact ratio at $n=50$, 1.2 percent at $n=100$, and 0.5 percent at $n=200$, and agrees with it to two decimals from $n=500$. The expansion shows two things. First, the inflation is a finite-$n$ effect, of order $n^{-1}$ at fixed density, while at the extreme thresholds of Corollary~\ref{cor:ppp_consequences} both expected counts vanish. Second, the cubic term favors $abc>0$: under the uniform law, balanced triangles, those with an even number of negative correlations, outnumber unbalanced ones by ratios ranging from about $1.75$ to $15{,}300$ across Table~\ref{tab:triangles}. As noted after Lemma~\ref{lem:pairwise}, the entries along any forest are mutually independent, so this dependence first appears on cycles, the smallest of which is the triangle.

\begin{table}[htbp!]
\centering
\begin{tabular}{@{}rrrrrr@{}}
\toprule
$n$ & density & $\tau$ & $q_n/p_n^3$ & expansion & balanced/unbalanced \\
\midrule
$50$   & $0.01$ & $0.354$ & $4.95$ & $5.05$ & $15{,}300$ \\
$50$   & $0.03$ & $0.301$ & $2.53$ & $2.61$ & $290$ \\
$50$   & $0.10$ & $0.231$ & $1.38$ & $1.38$ & $16.1$ \\
$100$  & $0.01$ & $0.254$ & $3.01$ & $3.05$ & $275$ \\
$100$  & $0.03$ & $0.215$ & $1.79$ & $1.81$ & $35.1$ \\
$100$  & $0.10$ & $0.164$ & $1.19$ & $1.19$ & $6.38$ \\
$200$  & $0.01$ & $0.181$ & $2.02$ & $2.03$ & $39.3$ \\
$200$  & $0.03$ & $0.153$ & $1.40$ & $1.41$ & $10.9$ \\
$200$  & $0.10$ & $0.116$ & $1.10$ & $1.10$ & $3.58$ \\
$500$  & $0.01$ & $0.115$ & $1.41$ & $1.41$ & $9.24$ \\
$500$  & $0.03$ & $0.097$ & $1.16$ & $1.16$ & $4.34$ \\
$500$  & $0.10$ & $0.073$ & $1.04$ & $1.04$ & $2.21$ \\
$1000$ & $0.01$ & $0.081$ & $1.21$ & $1.21$ & $4.71$ \\
$1000$ & $0.03$ & $0.069$ & $1.08$ & $1.08$ & $2.79$ \\
$1000$ & $0.10$ & $0.052$ & $1.02$ & $1.02$ & $1.75$ \\
\bottomrule
\end{tabular}
\caption{Expected triangle count under the uniform law relative to the Erd\H{o}s--R\'enyi benchmark at the same edge density, $q_n(\tau)/p_n(\tau)^3$, computed by quadrature; $\tau$ is the threshold with $p_n(\tau)$ equal to the stated density, the expansion column is $1+\tfrac{1}{2(n+1)}(2z\phi(z)/p)^3$ with $z=\Phi^{-1}(1-p/2)$, and the last column is the ratio of balanced to unbalanced triangles under the uniform law.\label{tab:triangles}}
\end{table}

\FloatBarrier

\section{Simulation Evidence on the Repair Fraction}\label{sec:ncmsim}

Theorem~\ref{thm:nearest} gives $R_n=\|\tilde{C}-C^{*}\|_F^2/\|\tilde{C}-I_n\|_F^2\rightarrow1$ for the case $C_{0,n}=I_n$ simulated here. The distance to the positive semidefinite cone identifies how the two constraints share this limit: the Frobenius projection onto the cone truncates negative eigenvalues, so the cone accounts for the fraction $P_n=\sum_i\min\{\lambda_i(\tilde{C}),0\}^2/\|\tilde{C}-I_n\|_F^2$, and writing $\lambda_i(\tilde{C})=1+\sigma\sqrt{n}s_i$, the semicircle law for the $s_i$ \citep{BaiSilverstein:2010} gives $P_n\rightarrow\int\min\{s,0\}^2\rho_{\mathrm{sc}}(s)ds=\tfrac12$, where $\rho_{\mathrm{sc}}(s)=\frac{1}{2\pi}\sqrt{4-s^2}$: the cone relaxation attains half of the limiting repair fraction, and the increment from imposing the unit diagonal accounts for the other half. Table~\ref{tab:ncm} reports $R_n$ and $P_n$ for two entry distributions; the nearest correlation matrix $C^{*}$ is computed by the alternating-projections algorithm of \citet{Higham:2002}.

Two features are visible. First, both ratios approach their limits slowly, consistent with corrections of order $(\sigma\sqrt{n})^{-1}$: at $n=400$ the Rademacher specification has $1-R_n=0.143$ against $\sigma\sqrt{n}=20$. Second, the gap between $R_n$ and $P_n$ widens with $n$, as it must: the columns converge to different limits, and the gap is the increment attributable to the unit-diagonal constraint.

For a general target, the repair fraction $R_n=\|\tilde{C}_n-C_n^{*}\|_F^2/\|E_n\|_F^2$ of Section~\ref{sec:nearest} coincides with the tabulated ratio at $C_{0,n}=I_n$ but has the unobservable denominator $\|E_n\|_F^2=\|\tilde{C}_n-C_{0,n}\|_F^2$ in general. Under the i.i.d.\ model of Theorem~\ref{thm:nearest}, $\|E_n\|_F^2/(2d\sigma^2)\rightarrow1$ almost surely, so for any consistent estimator $\hat{\sigma}_n^2$ the plug-in ratio $\hat{R}_n=\|\tilde{C}_n-C_n^{*}\|_F^2/(2d\hat{\sigma}_n^2)$ satisfies $\hat{R}_n-R_n\rightarrow0$ in probability. The projection inequality stated after Theorem~\ref{thm:nearest} gives $\|C_n^{*}-C_{0,n}\|_F^2\leq(1-R_n)\|E_n\|_F^2$, and since $1-R_n\leq(1-\hat{R}_n)_{+}+|\hat{R}_n-R_n|$, it follows that $\|C_n^{*}-C_{0,n}\|_F/\|E_n\|_F\leq\sqrt{(1-\hat{R}_n)_{+}}+o_p(1)$. This is a model-based asymptotic interpretation of the computed repair, with an unspecified error term, and not a calibrated finite-sample certificate: $\hat{R}_n$ can exceed one, and a certificate would require accounting for the uncertainty in the denominator.

Theorem~\ref{thm:nearest} also covers projection with a spectral floor. If $C_{0,n}$ satisfies the floor $x$, its proof applies verbatim to projection onto $\{C\in\mathcal{E}_n:\lambda_{\min}(C)\geq x\}$, which inherits convexity, feasibility of $C_{0,n}$, and the nuclear-norm diameter bound; otherwise the projection carries an additional deterministic approximation bias.

\begin{table}[htbp!]
\centering
\begin{tabular}{@{}lrrrr@{}}
\toprule
& \multicolumn{2}{c}{Rademacher, $\sigma^2=1$} & \multicolumn{2}{c}{Uniform$[-1,1]$, $\sigma^2=1/3$}\\
\cmidrule(lr){2-3}\cmidrule(lr){4-5}
$n$ & $R_n$ & $P_n$ & $R_n$ & $P_n$ \\
\midrule
$50$  & 0.667 & 0.387 & 0.510 & 0.317 \\
$100$ & 0.746 & 0.420 & 0.614 & 0.366 \\
$200$ & 0.808 & 0.442 & 0.702 & 0.402 \\
$400$ & 0.857 & 0.458 & 0.773 & 0.430 \\
\midrule
$n\rightarrow\infty$ & $1$ & $\tfrac12$ & $1$ & $\tfrac12$ \\
\bottomrule
\end{tabular}
\caption{Realized repair fraction $R_n$ and the PSD ratio $P_n$, averaged over $20$ independent replications for $n\leq200$ and $5$ for $n=400$; Monte Carlo standard errors of the reported averages are below $0.002$ throughout. Theorem~\ref{thm:nearest} gives $R_n\rightarrow1$, while $P_n$ converges to $\tfrac12$, so the gap between the columns is the increment attributable to the unit-diagonal constraint.\label{tab:ncm}}
\end{table}

\FloatBarrier
\section{On the Centering in the Log-Determinant CLT}\label{sec:centering}

For $C_n$ uniformly distributed on $\mathcal{E}_n$, \citet[theorem~3]{HaneaNane:2018} establish asymptotic normality of $\log\det C_n$ using the Beta factorization $\log\det C_n=\sum_{j=1}^{n-1}\log B_j$, $B_j\sim\operatorname{Beta}(\tfrac{j+1}{2},\tfrac{n-j}{2})$, which underlies Proposition~\ref{prop:WishartGram}(i), via a Lyapunov CLT for triangular arrays. With the exact mean $\mathbb{E}[\log\det C_n]=-n+\tfrac{1}{2}\log(2n)+\tfrac{1+\gamma}{2}+o(1)$ from the digamma expression in Section~\ref{sec:LKJ}, the result reads
$$
\frac{\log\det C_n + n - \tfrac{1}{2}\log(2n) - \tfrac{1+\gamma}{2}}{\sqrt{2\log(n/2)}} \overset{d}{\rightarrow} N(0,1),
$$
where the $-n$ centering comes from the bulk, since $\int_0^4\log(x)f_{\mathrm{MP}}(x)dx=-1$. The centering displayed in \citet[theorem~3]{HaneaNane:2018} is $-\bigl(n+\tfrac12\log(n/2)+\tfrac{\gamma}{2}+1\bigr)$, which differs in the sign of its logarithmic, $\gamma/2$ and constant terms.

A Lyapunov CLT centers the sum $\log\det C_n=\sum_{j=1}^{n-1}\log B_j$ at its exact mean, the digamma expression above. Two observations support the reading that the display is a misprint. First, it is inconsistent with the mean that \citet{HaneaNane:2018} themselves derive, namely $-n+\tfrac12\log(n/2)+1+\tfrac{\gamma}{2}+o(1)$: relative to that expression, the centering in their theorem~3 has the wrong sign on the logarithmic term, on $\gamma/2$ and on the constant, that is, on the entire lower-order bracket. Second, their own simulations agree with the exact mean rather than with the displayed centering: their standardized means of about $1.27$ for $(\log\det C_n+n)/\sqrt{2\log(n/2)}$ at $n=400,500$ match the $1.269$ and $1.277$ predicted by the centering above, whereas their derived mean predicts $1.210$ and $1.219$ and the displayed centering predicts the opposite sign. The gap between the exact mean and the expression derived by \citet{HaneaNane:2018} is $\log2-\tfrac12\approx0.193$, which is immaterial for the CLT because it is $o(\sqrt{\log n})$; the sign of the logarithmic term is not.

A related centering appears in the subsequent literature: \citet{ParolyaHeinyKurowicka:2024} observe that their CLT for sample correlation matrices reproduces the displayed centering upon setting $p=n$, but that specialization describes the ensemble with $T=n$ observations, for which $\mathbb{E}[\log\det C_n]=-n-\tfrac12\log n+O(1)$ is indeed correct; the uniform law corresponds to $T=n+1$, where the logarithmic term enters with a positive sign.

\end{document}